\documentclass[12pt]{amsart}

\title{Universal minimal flow of the homeomorphism group of the Lelek fan}

\usepackage{amsfonts,amssymb,amsmath,amsthm}
\usepackage{url}
\usepackage{xcolor}
\usepackage{enumerate}
\usepackage{comment}
\usepackage{graphics}
\usepackage{graphicx}
\usepackage{eqnarray}
\usepackage{bbm}
\usepackage{anysize}
\usepackage{mathtools}

\marginsize{1 in}{1 in}{1 in}{1 in}

\theoremstyle{plain}
\newtheorem{thm}{Theorem}[section]
\newtheorem{lemma}[thm]{Lemma}
\newtheorem{cor}[thm]{\bf Corollary}

\newtheorem{rem}[thm]{\bf Remark}

\theoremstyle{prop}
\newtheorem{prop}[thm]{\bf Proposition}

\newtheorem{question}[thm]{\bf Question}

\newcommand{\age}{{\rm Age}}
\def\acts{\curvearrowright}
\def\aut{{\rm Aut}}
\def\Exp{{\rm Exp}}
\def\mesh{{\rm mesh}}
\def\spr{{\rm spr}}
\def\Clop{{\rm Clop}}
\def\lel{\mathbb{L}}
\def\f{\mathcal{F}}
\def\g{\mathcal{G}}
\def\G{\mathbb{G}}
\def\c{\mathcal{C}}
\def\ch{\mathcal{C}}
\def\d{\mathcal{D}}

\def\ll{\mathcal{L}}

\def\u{\mathcal{U}}
\def\fin{{\rm{ FIN }}}
\def\supp{{\rm{ supp }}}
\def\rest{\restriction}
\def\n{\mathbb{N}}
\def\Id{{\rm{Id}}}

\newcommand{\mc}[1]{\mathcal{#1}}

\title{The universal minimal flow of the homeomorphism group of the Lelek fan} 

\author[D. Barto\v{s}ov\'{a}]{Dana Barto\v{s}ov\'{a}}
\address{Institute de Matematica e Estat\'istica, Universidade de S\~ao Paulo, Brazil  {\em AND}
Department of Mathematical Sciences, Carnegie Mellon University, Pennsylvania, USA}
\email{dbartoso@andrew.cmu.edu}

\author[A. Kwiatkowska]{Aleksandra Kwiatkowska}
\address{Institut f\"{u}r Mathematische Logik und Grundlagenforschung, Universit\"{a}t  M\"{u}nster, Einsteinstrasse 62,
48149  M\"{u}nster,  Germany {\em AND}
Instytut Matematyczny, Uniwersytet Wroc{\l}awski,  pl. Grunwaldzki 2/4, 50-384 Wroc{\l}aw, Poland}
\email{kwiatkoa@uni-muenster.de}

\subjclass[2010]{05D10, 37B05, 54F15, 03C98}  

\keywords{Dual Ramsey Theorem, universal minimal flows, homeomorphism group of the Lelek fan, Fra\"{i}ss\'{e} limits}

\begin{document}

\maketitle

\begin{abstract}
We compute the universal minimal flow of the homeomorphism group of the Lelek fan -- a one-dimensional tree-like continuum with many symmetries.
\end{abstract}

\section{Introduction}
Let $G$ be a topological group and let $X$ be a compact  space. A continuous action $G\acts X$ is called 
a {\em  $G$-flow}  (or just a flow, if the group $G$ is understood from  the context). A \emph{$G$-map} between two flows $G\acts X$ and $G\acts Y$ is a map $f:X\to Y$ such that for every $g\in G$ and $x\in G$ we have $f(gx)=gf(x).$
A flow is  {\em minimal} if all of its orbits are
 dense. It is a general result in topological dynamics, due to Ellis, that for any topological group $G$ there is a  
 {\em universal minimal flow} $M(G)$, that is, a minimal flow $G\acts M(G)$ such that for any minimal flow $G\acts X$
 there is  a continuous $G$-map from $M(G)$ onto $X$.  For a compact group $G$, the flow $G\acts M(G)$
 can be identified with the  action of $G$ on itself by left translations. If $G$ is locally compact, but not compact 
 (such as $G$ discrete) $M(G)$ is a very large space, in particular, it is always non-metrizable. For example, when $G$
 is the set of integers $\mathbb{Z}$, $M(G)$ is the Gleason space of $2^{\mathfrak{c}},$ that is the Stone space of the Boolean algebra of regular open sets of $2^{\mathfrak{c}},$ where $\mathfrak{c}$ is the cardinality of real numbers.

 Many groups that have been studied by descriptive set-theorists and model-theorists, as it turned out in the last
 10-20 years,  have metrizable universal minimal flows
 that can be computed explicitly and a surprising number of them have a trivial universal minimal flow, we call such groups {\em extremely amenable}.
Pestov  \cite{P} applied the finite Ramsey theorem  
to show that  $\aut(\mathbb{Q},<)$, the group of order-preserving bijections of rationals with the pointwise convergence topology,
is  extremely amenable. 
Glasner and Weiss \cite{GW}
  used the finite Ramsey theorem to identify the universal minimal flow of $S_\infty$, the  group of all permutations of
  natural numbers $\mathbb{N}$, with its canonical action on the space of all linear orderings on $\mathbb{N}$.
   In 2005, Kechris, Pestov, and Todorcevic \cite{KPT}  developed a general powerful tool to compute  universal minimal flows
  of automorphism groups of countable model-theoretic structures via establishing a strong connection between
  the dynamics of such groups and the structural Ramsey theory.

 The focus of our paper is on homeomorphism groups of compact spaces. 
We compute the universal minimal flow of the homeomorphism group $H(L)$ of the Lelek fan $L$.
We will show that it is equal to the action of $H(L)$ on the space of all maximal chains on $L$ consisting of continua
 containing the top point of $L$, which is induced from the evaluation action of $H(L)$ on $L$.

An important motivation for our work was the following (still open) question due to Uspenskij from 2000.
  \begin{question}[Uspenskij, \cite{U}]
Let $P$ be the pseudo-arc and let $H(P)$ be its homeomorphism group.
What is the the universal minimal flow of $H(P)$? In particular,
 is the action $H(P)\acts P$ given by $(h,x)\to h(x)$ the universal minimal flow of $H(P)$?
\end{question} 

Both the pseudo-arc and the Lelek fan are well known very homogeneous continua that can be constructed as natural quotients of projective  Fra\"{i}ss\'{e}  limits
of finite structures.
 
 Pestov~\cite{P} showed as a consequence of the Ramsey theorem that 
 the group of increasing homeomorphisms of the unit interval is extremely amenable 
 and identified  the universal minimal flow of the orientation preserving homeomorphisms of the circle $S^1$ with its natural action on $S^1.$ 
 Glasner and Weiss \cite{GW2} proved that the universal minimal flow
 of the homeomorphism group of the Cantor set is the
action on the space of maximal chains of closed subsets of the Cantor set, which is induced from the evaluation action.
These seem to be the only examples of homeomorphism groups for which the  universal minimal flow was computed. In  each of these examples, 
the description of the universal minimal flow  follows directly from 
a description of the universal minimal flow of a certain automorphism group of a countable structure (rational numbers in the case of $S^1$ and $[0,1]$,
and the countable atomless Boolean algebra in the case of the Cantor set).
Universal minimal flows of homeomorphism groups of many very simple compact spaces, such as $[0,1]^2$
or the sphere $S^2$, are unknown.

Unlike for the automorphism groups of countable structures, there are no general  techniques  to 
compute universal minimal flows of homeomorphism groups of compact spaces. To obtain the universal minimal flow of 
$H(L)$, we will use our earlier construction, presented in~\cite{BK}, of the Lelek fan $L$ as a  quotient of a projective
 Fra\"{i}ss\'{e}  limit $\lel$. First, we compute the universal minimal flow of the  automorphism group $\aut(\lel)$;
 we will use tools provided by Kechris, Pestov, and Todorcevic \cite{KPT} and a new Ramsey theorem, which we prove using the  Dual Ramsey theorem \cite{GR}. 
Second, by relating  $\lel$ and $L$, we compute the universal minimal flow of $H(L)$.
This second step is novel and nontrivial, and we hope it will find applications to homeomorphism groups of other compact spaces.


\section{Discussion of results} 


A {\em continuum} is a compact connected  metric space. Denoting by $C$ the Cantor set and by $[0,1]$ the unit interval,
the {\em Cantor fan} is the 
quotient of $C\times [0,1]$ 
by the equivalence relation $\sim$ given by $(a,b)\sim (c,d)$ if and only if either $(a,b)=(c,d)$ or $b=d=0.$
For a continuum $X,$ a point $x\in X$ is an {\em endpoint} in $X$ if for every homeomorphic embedding $h:[0,1]\to X$ with $x$ in the image of $h$ either $x=h(0)$ or $x=h(1).$ The {\em Lelek fan} $L$, constructed by Lelek  in \cite{L}, can be characterized as the unique non-degenerate subcontinuum of the Cantor fan whose endpoints are dense in $L$ (see \cite{BO} and \cite{C}).
Denote by $v$ the {\em top} (which we will also sometime call the {\em root}) $(0,0)/\!\!\sim$ of the Lelek fan.

If $K$ is a compact  topological space, a  chain $\ch$ on $K$ is  a family of closed subsets of $K$ such that
for every $C_1, C_2\in\ch$, either $C_1\subset C_2$ or $C_2\subset C_1$. We say that a chain $\ch$ 
is  maximal if for every closed set
$C\subset K$, if $\{C\}\cup\ch$ is a chain then $C\in\ch$.
The set ${\rm Exp}(K)$ of all closed subsets of a compact  topological space $K$  equipped with the Vietoris topology is a compact space, which we introduce in 
Section~\ref{sec:chain}.

Let $Y^*\subset {\rm Exp (Exp}(K))$ be the space of all maximal chains $\ch$ on $L$ such that each $C\in\ch$ is connected and it contains the root of $L$.
This space is compact, which we prove in Proposition~\ref{ccompact}, and 
the natural action of $H(L)$ -- the homeomorphism group of the Lelek fan $L$ on $L$ given by $(g,x)\to g(x)$ induces an action on ${\rm Exp}(L)$ which further induces an action on ${\rm Exp (Exp}(L))$ 
which is invariant on $Y^*$. The main result of this article is the following:
\begin{thm}\label{thumf}
The universal minimal flow of $H(L)$ -- the homeomorphism group of the Lelek fan $L$ -- is 
\[
H(L)\acts Y^*.
\]
\end{thm}

To prove Theorem \ref{thumf}, we will first find a ``quotient'' description of the universal minimal flow of $H(L)$.
Let $H$ be the closed subgroup of $H(L)$  consisting of homeomorphisms that preserve the ``generic'' 
maximal chain in $Y^*$. 
Such a chain is constructed explicitly, we expand the projective Fra\"{i}ss\'e family $\f$ of finite fans, whose limit gives the Lelek
fan, to the projective Fra\"{i}ss\'e family $\f_c$ of finite  fans expanded by  a maximal chain of  connected sets containing the root. The limit of this new family gives the Lelek fan equipped
with the required ``generic'' chain. The details and necessary definitions are contained in the next  sections.
The quotient space $H(L)/H$ is precompact in the quotient of the right uniformity on $H(L)$ 
and consequently its completion
$\widehat{H(L)/H}$ is compact. The group $H(L)$ acts on itself by left translations. 
This actions induces an action on the quotient, which extends to the completion.
Theorem \ref{thumf} will follow from Theorem \ref{thumf2}.

\begin{thm}\label{thumf2}
The universal minimal flow of $H(L)$--the homeomorphism group of the Lelek fan $L$ is 
\[ 
H(L)\acts\widehat{H(L)/H}.
\]
\end{thm}

Let $\lel$ and $\lel_c$ be the projective Fra\"{i}ss\'{e} limits of $\f$ and $\f_c$ respectively and let
$ {\rm Aut}(\lel)$ and $ {\rm Aut}(\lel_c)$ be their automorphism groups.
In Section \ref{extamen}, we show that  ${\rm Aut}(\lel_c) $ is extremely amenable and in
 Section \ref{umfa}, we provide two equivalent descriptions of the universal minimal flow of $ {\rm Aut}(\lel)$.
We prove our main result in Section \ref{umfh}.

\section{Preliminaries}\label{prelim}

We first review the  Fra\"{i}ss\'{e} and the projective  Fra\"{i}ss\'{e} constructions, as well as the construction of the Lelek fan in the 
projective  Fra\"{i}ss\'{e} framework the authors introduced in \cite{BK} (Sections \ref{sec:fra}, \ref{pff}, and \ref{sec:lelek}). 
We then discuss topics specifically relevant to studying the universal minimal flow
of the homeomorphism group of the Lelek fan: maximal chains on compact spaces (Section \ref{sec:chain}),
 uniform spaces (Section \ref{pus}), and the Kechris-Pestov-Todorcevic correspondence for  Fra\"{i}ss\'{e}-HP families (Section \ref{sec:KPT}).

\subsection{Fra\"{i}ss\'{e} families}\label{sec:fra} 

Given a first-order language $\mathcal{L}$ that consists of relation symbols $r_i$,  with arity $m_i$,  $i\in I$, and function symbols $f_j$, 
  with arity $n_j$, $j\in J$, and two structures $A$ and $B$ in $\ll$, say that $i: A\to B$ is an {\em embedding} if it is an injection such that
 for a function symbol $f$ in $\mathcal{L}$ of arity $n$ and $x_1,\ldots,x_n\in A$ we have
$i( f^A(x_1,\ldots,x_n))=f^B(i(x_1),\ldots,i(x_n))$;
and for a relation symbol $r$ in $\mathcal{L}$ of arity $m$ and $x_1,\ldots,x_m\in A$  we require
$ r^A(x_1,\ldots,x_m)$ iff $r^B(i(x_1),\ldots,i(x_m))$. 
 For a relation symbol $r\in\mathcal{L}$ with arity $k$ and a function $f: A\to B$, say that $f$ is {\em $R$-preserving}
 if for every $x_1,\ldots,x_k\in A$  we have
$ r^A(x_1,\ldots,x_k)$ iff $r^B(f(x_1),\ldots,f(x_k))$.

  A countable first order  structure $M$  {in $\ll$} is  {\em locally finite} if every finite subset of $M$ generates a finite substructure. It is
  {\em ultrahomogeneous} if every isomorphism between finite
 substructures of $M$ can be extended to an automorphism of $M$.  In that case, $\f=\age(M)$, the family of all
 finite substructures of $M$,
 has the  following three properties: the {\em hereditary property} {\em (HP)},
 that is, if $A\in\f$ and $B$ is a substructure of $A$, then $B\in\f$; the {\em  joint embedding property} {\em (JEP)}, that is, for any $A,B\in\f$ there is $C\in\f$
 such that $A$ embeds both into $B$ and into $C$;
 and the {\em  amalgamation property} {\em (AP)}, that is, for any $A,B_1,B_2\in\f$,
 any embeddings $\phi_1: A\to B_1$ and $\phi_2: A\to B_2$, there exist $C\in\f$ and embeddings
 $\psi_1: B_1\to C$ and $\psi_2: B_2\to C$ such that $\psi_1\circ\phi_1=\psi_2\circ\phi_2$.
 Conversely, by a classical theorem due to Fra\"{i}ss\'{e}, if a countable family of finite structures $\f$ in some language $\ll$
 has the HP, the JEP and the  AP,  
 then there is a unique countable locally finite ultrahomogeneous structure $M$
 such that $\f=\age(M)$. 
 
 In this paper, we will call a countable family of finite structures that satisfies the JEP and the AP a {\em Fra\"{i}ss\'{e}-HP family}
(read as Fra\"{i}ss\'{e} minus HP family), 
  and   a countable family of finite structures that satisfies the HP, the JEP, and the AP we will call a {\em Fra\"{i}ss\'{e} family}.
 
 A {\em Fra\"{i}ss\'{e} limit} of a Fra\"{i}ss\'{e} family $\f$   is a countable locally finite ultrahomogeneous structure $M$ such that  $\f=\age(M)$,
 and 
a {\em Fra\"{i}ss\'{e} limit} of a Fra\"{i}ss\'{e}-HP family $\f$ is a countable structure $M$ such that every structure in $\f$ embeds into $M$,
for every finite subset $X$ of $M$ there is $A\in\f$ and an embedding $i:A\to M$ such that $X\subset i(A)$, and $M$ is ultrahomogeneous with respect to $\f$, that is, 
every isomorphism between finite substructures of $M$ which are isomorphic to a structure in $\f$ can be extended to an automorphism of $M$.
Clearly every Fra\"{i}ss\'{e}  family is also a Fra\"{i}ss\'{e}-HP family.
If $\f$ is a Fra\"{i}ss\'{e} family or it is a Fra\"{i}ss\'{e}-HP family,    the Fra\"{i}ss\'{e} limit of $\f$ always exists and it is unique up to an isomorphism.

For example, the rationals with the usual ordering is the Fra\"{i}ss\'{e} limit of the family of finite linear orders,
the Rado graph is  the Fra\"{i}ss\'{e} limit of the family of finite graphs, and the countable atomless Boolean algebra is the
 Fra\"{i}ss\'{e} limit of the family of finite Boolean algebras.

We say that a  family $\g_1$ is {\em cofinal } in a  family $\g_2$ if for every $A\in\g_2$ there  are $B\in\g_1$ and an embedding $\phi: A\to B$.
\begin{rem}\label{cofinall}
{\rm Suppose that a family $\g_1$ is contained in and cofinal in a Fra\"{i}ss\'{e}-HP  family $\g_2$. Then
$\g_1$ is also a Fra\"{i}ss\'{e}-HP  family and moreover  Fra\"{i}ss\'{e}  limits of $\g_1$ and of $\g_2$ are isomorphic. }
\end{rem}

\subsection{Projective Fra\"{i}ss\'{e} families}\label{pff}

Given a first-order language $\mathcal{L}$ that consists of relation symbols $r_i$,  with arity $m_i$,  $i\in I$, and function symbols $f_j$, 
  with arity $n_j$, $j\in J$,
a \emph{topological $\mathcal{L}$-structure} is a compact zero-dimensional second-countable space $A$ equipped with
closed (in the product topology) relations $r_i^A\subset A^{m_i}$ and continuous functions $f_j^A: A^{n_j}\to A$, $i\in I, j\in J$.
A~continuous surjection $\phi: B\to A$ between two topological $\mathcal{L}$-structures is an
 \emph{epimorphism} if it preserves the structure, that is, for a function symbol $f$ in $\mathcal{L}$ of arity $n$ and $x_1,\ldots,x_n\in B$ we require:
\[
 f^A(\phi(x_1),\ldots,\phi(x_n))=\phi(f^B(x_1,\ldots,x_n));
\]
and for a relation symbol $r$ in $\mathcal{L}$ of arity $m$ and $x_1,\ldots,x_m\in A$  we require:
\begin{equation*}
\begin{split}
&  r^A(x_1,\ldots,x_m) \\ 
&\iff \exists y_1,\ldots,y_m\in B\left(\phi(y_1)=x_1,\ldots,\phi(y_m)=x_m, \mbox{ and } r^B(y_1,\ldots,y_m)\right).
\end{split}
\end{equation*}
By an \emph{isomorphism}  we mean a bijective epimorphism.

Let $\mathcal{G}$ be a countable family of finite topological $\mathcal{L}$-structures. We say that $\mathcal{G}$ is a \emph{ projective Fra\"{i}ss\'{e} family}
if the following two conditions hold:

(JPP) (the joint projection property) for any $A,B\in\mathcal{G}$ there are $C\in \mathcal{G}$ and epimorphisms from $C$ onto $A$ and from $C$ onto $B$;

(AP) (the amalgamation property) for $A,B_1,B_2\in\mathcal{G}$ and any epimorphisms $\phi_1: B_1\to A$ and $\phi_2: B_2\to A$, there exists $C\in\mathcal{G}$ with epimorphisms
 $\psi_1: C\to B_1$ and $\psi_2: C\to B_2$ such that $\phi_1\circ \psi_1=\phi_2\circ \psi_2$.

A topological $\mathcal{L}$-structure $\mathbb{G}$ is a \emph{ projective Fra\"{i}ss\'{e} limit } of  a projective Fra\"{i}ss\'{e} family $\mathcal{G}$ if the following three conditions hold:

(L1) (the projective universality) for any $A\in\mathcal{G}$ there is an epimorphism from $\mathbb{G}$ onto~$A$;

(L2) for any finite discrete topological space $X$ and any continuous function
 $f: \mathbb{G} \to X$ there are $A\in\mathcal{G}$, an epimorphism $\phi: \mathbb{G}\to A$, and a function
$f_0: A\to X$ such that $f = f_0\circ \phi$;

(L3) (the projective ultrahomogeneity) for any $A\in \mathcal{G}$ and any epimorphisms $\phi_1: \mathbb{G}\to A$ 
and $\phi_2: \mathbb{G}\to A$
there exists an isomorphism $\psi: \mathbb{G}\to \mathbb{G}$ such that $\phi_2=\phi_1\circ \psi$.

\begin{rem}\label{coveri}
{\rm It follows from (L2) above that if  $\mathbb{G}$ is the projective Fra\"{i}ss\'{e} limit of $\mathcal{G}$, then every finite open cover can be {\em refined by an epimorphism, i.e. for every open cover
 $\mathcal{U}$}
of $\mathbb{G}$  
there is an epimorphism  $\phi:\mathbb{G}\to A$,  for some  
  $A\in\mathcal{G}$, such that for every $a\in A$, $\phi^{-1}(a)$ is contained in an open set in $\mathcal{U}$. }
\end{rem}

\begin{thm}[Irwin-Solecki, \cite{IS}]\label{is}
 Let $\mathcal{G}$ be a projective Fra\"{i}ss\'{e} family of finite topological $\mathcal{L}$-structures. Then:
\begin{enumerate}
 \item there exists a projective Fra\"{i}ss\'{e} limit of $\mathcal{G}$;
\item any two projective Fra\"{i}ss\'{e} limits of $\mathcal{G}$ are isomorphic.
\end{enumerate}
\end{thm}

The theorem below is a folklore, nevertheless it has not been published. 
 It says that the projective Fra\"{i}ss\'{e} theory is a special case of an (injective) Fra\"{i}ss\'{e} theory via a generalization of Stone duality.
 \begin{thm}\label{dual}
For a  projective Fra\"{i}ss\'{e} family $\f$ in a relational language with a projective Fra\"{i}ss\'{e} limit $\mathbb{F}$
 there is an equivalent   via a contravariant functor  
 (defined on $\f\cup\{\mathbb{F}\}$ and on all epimorphisms between structures in $\f\cup\{\mathbb{F}\}$)
 Fra\"{i}ss\'{e}-HP family $\g$ with a  Fra\"{i}ss\'{e} limit $\mathbb{G}$. 
\end{thm}

 Theorem \ref{dual} 
 will follow from Proposition \ref{stone1},  a generalization of 
 the classical Stone duality between Boolean algebras with embeddings and compact totally disconnected  spaces
 with continuous surjections, which we recall here. In this paper, we will not consider $\f$ in a language that contains function symbols.
  \begin{prop}\label{stone0}
  The family of compact totally disconnected  spaces $\f_0$  with continuous surjections  is equivalent via a contravariant functor to
the family $\g_0$ of Boolean algebras with embeddings.  
  \end{prop}
 In the Stone duality, to $K\in\f_0$ we associate the Boolean algebra $\Clop(K)$ of clopen sets of $K$ 
with the usual operations of the union $\cup^{\Clop(K)}$, 0 is the empty set and 1 is identified with $K$, the intersection  $\cap^{\Clop(K)}$ and the complement $^{-\Clop(K)}$, 
 and to a continuous surjection $f:L\to K$ we associate 
 an embedding $F:\Clop(K)\to \Clop(L)$ given by $F(X)=f^{-1}(X)$.

 \begin{prop}\label{stone1}
 Let $\ll$ be a relational language and let $\f_1$ be a family of topological $\ll$-structures, maps between structures are epimorphisms.
 Then there is a family $\g_1$ of countable structures in the language equal to the union of the language of Boolean algebras and of $\ll$,
maps between structures are embeddings,
 such that $\f_1$ is equivalent to $\g_1$ via a contravariant functor.
 \end{prop} 
\begin{proof}
Let $\ll= \{R_1,\ldots, R_n\}$, where $R_i$ is a relation symbol of the arity $m_i,$ be the language of $\f_1.$
Let $\ll'= \{S_1,\ldots, S_n,\cup,\cap,^-,0,1\}$ be the language where $S_i$  is a relation symbol of the arity $m_i$
and $\{\cup,\cap,^-, 0,1\}$ is the language of Boolean algebras.
For $K=(K,R_1^K,\ldots, R_n^K)\in\f_1$, let $M=(M, S_1^M, \ldots, S^M_n, \cup^M, \cap^M, ^{-M}, 0^M,1^M\}$ be the structure
such that $M=\Clop(K)$ is  the family of all clopen sets of $K$,  $\cup^M$ is the union, $\cap^M$ is the intersection,   $^{-M}$ is the complement,
$0^M$ is the empty set and $1^M=M$.
Moreover, we require that for every~$i$, $S_i^M(X_1,\ldots, X_{m_i})$ 
  iff
  for some $c_1\in X_1,\ldots, c_{m_1}\in X_{m_1}$,
  we have $R_i^K(c_1,\ldots, c_{m_i})$.

Let $\g_1$ be the family of all $M$'s obtained in this way from a $K\in\f_1$ with 
the 
embeddings.

Let $f:L\to K$, where $K,L\in\f_1$, be a continuous surjection and let $F:\Clop(K)\to \Clop(L)$  be the  map given by $F(X)=f^{-1}(X)$.
In view of Proposition \ref{stone0}, all we have to check is that $f$ is $R_i$-preserving if and only if 
$F$  is $S_i$-preserving, and that will follow from the two claims below.

\smallskip

\noindent {\bf{Claim.}}  If  $f$ is $R_i$-preserving then $F$ is $S_i$-preserving.
\begin{proof}
If  $S_i^{\Clop(K)}(X_1,\ldots, X_{m_i})$ 
then  $R_i^K(a_1,\ldots, a_{m_i})$ for some $a_i\in X_i$,  
which implies that  for some $c_i\in f^{-1}(a_i)\subset f^{-1}(X_i)$, $R_i^L(c_1,\ldots, c_{m_i})$, 
hence  $S_i^{\Clop(L)}(f^{-1}(X_1),\ldots, f^{-1}(X_{m_i}))$, i.e. $S_i^{\Clop(L)}(F(X_1),\ldots, F(X_{m_i})).$
 
 Conversely,
if $S_i^{\Clop(L)}(F(X_1),\ldots, F(X_{m_i}))$, that is $S_i^{\Clop(L)}(f^{-1}(X_1),\ldots, f^{-1}(X_{m_i}))$
then for some $c_i\in f^{-1}(X_i)$, $R_i^L(c_1,\ldots, c_{m_i})$, therefore 
 $R_i^K(f(c_1),\ldots, f(c_{m_i}))$, which gives
 $S_i^{\Clop(K)}(X_1,\ldots, X_{m_i})$. 
 
\end{proof}

 \smallskip
 
 \noindent {\bf{Claim.}}
 If  $F$ is $S_i$-preserving then $f$ is $R_i$-preserving.
 \begin{proof}
We have $R_i^{K}(a_1,\ldots, a_{m_i})$  iff 
for every $X_i\in\Clop(K)$ such that $a_i\in X_i$ we have $S_i^{\Clop(K)}(X_1, \ldots,X_{m_i} )$   iff 
for every $X_i\in\Clop(K)$ such that $a_i\in X_i$ we have $S_i^{\Clop(L)}(f^{-1}(X_1), \ldots, f^{-1}(X_{m_i}))$ 
iff 
for every $X_i\in\Clop(K)$ such that $a_i\in X_i$ there exist  $c_i\in f^{-1}(X_i)$
for which we have $R_i^L(c_1,\ldots, c_{m_i})$
iff  there exist  $c_i\in f^{-1}(a_i)$
for which we have $R_i^L(c_1,\ldots, c_{m_i})$.
 In the first and last equivalences we used that the relations $R_i^K$ and $R_i^L$ are closed 
 in $K^{m_i}$ and $L^{m_i}$, respectively.

\end{proof} \end{proof}

Now one may ask why we study  projective Fra\"{i}ss\'{e} families at all.
The reason is that it is more natural to use  projective Fra\"{i}ss\'{e} families to construct and study compact spaces, like the Lelek fan or the pseudo-arc,
rather than to study them via families of finite Boolean algebras equipped with  relations.

In further sections, we will introduce families $\f_c$ and $\f_{cc}$ of finite  fans expanded by  an additional structure, which will be neither a function nor a relation,
for which we will have to prove an analog of Theorem \ref{dual}. 
These families will not exactly fall into the framework of the projective Fra\"{i}ss\'{e} theory discussed in this section.
Nevertheless, we will still call them projective Fra\"{i}ss\'{e} families, and their limits we will call projective Fra\"{i}ss\'{e} limits.

\subsection{Construction of the Lelek fan}\label{sec:lelek}

For completeness, we repeat here more or less Section 3.1 from \cite{BK2}, where we review the construction of the Lelek fan from  \cite{BK}.
Unlike in \cite{BK} and \cite{BK2}, 
we will not assume that all branches in a finite fan are of the same length.

By a {\em fan} we mean an undirected connected simple graph with all loops, with no cycles of the length greater than one, and
with a distinguished point $r$, called the {\em root}, such that all elements other than $r$ have degree at most 2. On a  fan $T,$ there is a natural partial tree order
$\preceq_T$: for $t,s\in T$ we let $s\preceq_T t$ if and only if $s$ belongs to the path connecting $t$ and the root.
We say that $t$ is a {\em successor} of $s$ if $s\preceq_T t$ and $s\neq t$.
It is an {\em immediate successor} if additionally there is no $p\in T$, $p\neq s,t$,  with $s\preceq_T p\preceq_T t$.
For a fan $T$ and $x,y\in T$ which are on the same branch and $x\preceq_T y$, by $[x,y]_{\preceq_T}$ we denote the interval
$\{z\in T: x\preceq_T z \preceq_T y\}$.

A {\em chain} in a fan $T$ is a subset of $T$ on which the order $\preceq_T$ is linear.
A {\em branch } of a fan $T$ is a maximal chain in  $(T,\preceq_T)$. 
If $b$ is a branch in $T$ with $n+1$ elements, we will sometimes  enumerate  
$b$ as $(b^0,\ldots,b^n)$, where $b^0$ is the root of $T$, and
$b^i$ is an immediate successor of $b^{i-1}$, for every $i=1, 2, \ldots, n$. In that case, $n $ will be called the {\em height} of the branch $b$.
Define the {\em height} of the fan to be the maximum of the heights of all of its branches
and define the {\em width} of the fan to be the number of its branches.

 Let $\mathcal{L}=\{R\}$ be the language with $R$ a binary relation symbol. 
For  a fan $T$ and $s,t\in T$, we let $R^T(s,t)$ if and only if $s=t$ or
 $t$ is an immediate successor of $s$.
 Let $\f$ be the family of all finite  fans, viewed as topological $\mathcal{L}$-structures, equipped with the discrete topology.

\begin{rem}
{\rm{ For two fans $(S,R^S)$ and $(T,R^T)$ in $\f$, a function $\phi: (S,R^S)\to (T,R^T)$ is an epimorphisms if and only if it is a surjective homomorphism, i.e., for every $s_1,s_2\in S$, $R^S(s_1,s_2)$ implies $R^T(\phi(s_1),\phi(s_2))$.
 }}
\end{rem}

 We say that a projective Fra\"{i}ss\'{e} family $\g_1$ is {\em coinitial } in a projective Fra\"{i}ss\'{e} family $\g_2$ if for every $A\in\g_2$ there are $B\in\g_1$ and an epimorphism $\phi: B\to A$.

\begin{prop}\label{Fraissef}
The family $\f$ is a projective Fra\"{i}ss\'{e} family.
 \end{prop}
 
In \cite[Proposition 2.3]{BK}, we proved that the family, which we call now  $\f_1$, of finite  fans with all branches  of the same length,
 is a projective Fra\"{i}ss\'{e} family. The proof of Proposition \ref{Fraissef} is essentially the same as the proof that $\f_1$ is a projective Fra\"{i}ss\'{e} family. By Theorem \ref{is}, there exists a unique projective  Fra\"{i}ss\'{e} limit of $\f$, which we denote by $\lel=(\lel, R^{\lel})$. The underlying set $\lel$ is homeomorphic to the Cantor set. The family $\f_1$ is coinitial in $\f$,  and this implies 
(by Remark \ref{cofinall} and Theorem \ref{dual}) that
the projective Fra\"{i}ss\'{e} limits of $\f$ and $\f_1$ are isomorphic. 
 Let $ R_S^{\lel}$ be the symmetrization of $ R^{\lel}$, that is, 
$ R_S^{\lel}(s,t)$ if and only if $ R^{\lel}(s,t)$ or $  R^{\lel}(t,s)$, for  $s,t\in\lel$.

\begin{thm}[Theorem 2.5, \cite{BK} ]
The relation $ R_S^{\lel}$ is an equivalence relation which has only one and two element equivalence classes.
\end{thm}

\begin{thm}[Theorem 2.6, \cite{BK} ]
The quotient space $\lel/R^{\lel}_S$ is homeomorphic to the Lelek fan $L.$
\end{thm}

Let $\pi:\lel\to L$ denote the quotient map given by $R^{\lel}_S$.
We denote by $\aut(\lel)$  the group of all automorphisms of $\lel$, that is, the group of all homeomorphisms of $\lel$ that preserve the relation $R$. 
This is a topological  group when equipped with the compact-open topology inherited from $H(\lel)$,
 the group of all homeomorphisms of the Cantor set underlying the structure $\lel$.
Since $R^\lel$ is closed in $\lel\times \lel$, the  group $\aut(\lel)$ is closed in $H(\lel)$.

 Let $\pi^*$ be the map that takes $h\in \aut(\lel)$ to $h^*\in H(L)$ and  $h^*\pi(x)= \pi h(x)$
for every $h\in \aut(\lel)$ and $x\in\lel$.
We will frequently identify $\aut(\lel)$ with the corresponding subgroup $\{h^*: h\in \aut(\lel)\}$ of $H(L)$. 
Observe  that the compact-open topology  on $\aut(\lel)$ is finer than the topology on $\aut(\lel)$ that is inherited from the compact-open 
topology on $H(L)$.

\subsection{Spaces of maximal chains}\label{sec:chain}

We will assume throughout  the paper that every compact space is Hausdorff.
Let $K$ be a compact  topological space. A {\em chain} $\c$ on $K$ is  a family of closed subsets of $K$ such that
for every $C_1, C_2\in\c$, either $C_1\subset C_2$ or $C_2\subset C_1$. 
Sometimes we will call  the sets in a chain {\em links}.
We say that a chain $\c$ 
is {\em maximal} if for every closed set
$C\subset K$, if $\{C\}\cup\c$ is a chain then $C\in\c$. 
Note that if $\c$ is a maximal chain and $A\subset \c$, then $\bigcap A\in\c$ and $\overline{\bigcup A}\in\c$.

The set ${\rm Exp}(K)$ of all closed subsets of  
$K$ is equipped with the Vietoris topology
generated by the sets
\[ [U_1,\ldots, U_n]=\{ F\in {\rm Exp}(K): F\subset U_1\cup\ldots\cup U_n { \rm\ and\ for\ every\ } 
i=1,\ldots, n, \ F\cap U_i\neq\emptyset\}, \] 
where $n\in\n$ and $U_1,\ldots, U_n$ are open in $K$.
Without loss of generality, $U_1,\ldots, U_n$ are only taken from some fixed basis of $K$.
If $K$ is metrizable by a metric $d_0$ then
the space ${\rm Exp}(K)$  is metrizable by the {\em Hausdorff metric} given by
\[d(X,Y)= \max \{ \sup_{x\in X} \inf_{y\in Y} d_0(x,y), \sup_{y\in Y} \inf_{x\in X} d_0(x,y) \}.
\]
It is not hard to show (see \cite[Lemma 6.4.7]{P2}) that every maximal chain is closed in the Vietoris topology.
It is well known that 
${\rm Exp}(K)$ is compact.
Uspenskij \cite{U} showed that 
the set of maximal chains on $K$ is closed 
in ${\rm Exp (Exp}(K))$, and therefore it is compact.

Let $\leq$ be a  partial order on $K$. We say that a  set $C\subset K$ is {\em downwards closed} if it is closed and for 
every $x, y\in K$, if $y\in C$ and $x\leq y$ 
then $x\in C$, and a chain $\c$ on $K$ be {\em downwards closed} if every $C\in\c$ is downwards closed.
A downwards closed chain $\c$ is {\em downwards closed maximal} if for every downwards closed set
$C\subset K$, if $\{C\}\cup\c$ is a 
chain then $C\in\c$.

For a map $f: Y\to X$ and a chain $\c$ on $Y$,  by $f(\c)$ we will denote the chain  $\{f(C): C\in\c\}$.
We start with the following observations.
 \begin{lemma}\label{image3}
Let $K,M$ be compact sets and let $f: M\to K$ be a continuous surjection. If $\c$ is a maximal chain in $M$,
then the chain $f(\c)$ is also maximal. 
\end{lemma}

\begin{proof}
Suppose a closed set $D\subset K$ is such that $\{D\}\cup f(\c)$ is a chain.
We will show that there is $J\in\c$ satisfying $f(J)=D$.
 Let  $K_1=\{ C\in f(\c): C\supset D\}$ and 
let $M_1=\{E\in\c: f(E)\in K_1\}.$
   As $\c$ is maximal, $M=\bigcap M_1\in \c$.
  Since $D\subset f(M)$, 
  we have that $J=f^{-1}(D)\cap M$ satisfies $f(J)=D$ and has the property that $\{J\}\cup\c$ is a chain,
  and hence by the maximality of $\c$, $J\in\c$   as required.
\end{proof}

Using Zorn's Lemma, we get the following.
\begin{lemma}
Let $K$ be a compact set and let $\mathcal{D}$ be a chain on $K$. Then there is a maximal chain on $K$ that extends $\mathcal{D}$.
\end{lemma}

Let $\f^*$ be the family of all topological $\ll$-structures that are countable inverse limits of finite  fans in $\f$. 
If $P$ is the  inverse limit of $(A_n, f_m^n)$, 
the relation $R^P$ on $P$ is defined as follows
\[ R^P(x,y) \text{ iff  for every } n, \  R^{A_n}(f^\infty_n(x), f^\infty_n(y)).
\]
Clearly we can identify $\f$ with a subset of $\f^*$ by assigning to $A$ the inverse limit of $(A, \Id_m^n)$.
Recall that $A\in\f$ is equipped with the tree partial order $\preceq_A$; we let $x\preceq_A y$ iff $x$ belongs to the segment
joining the root of $A$, $v_A$, with $y$. 
We let 
\[x\preceq_P y  \text{ iff  for every } n, \  f^\infty_n(x)\preceq_{A_n} f^\infty_n(y).\]

In particular, we have just defined a partial order on $\lel$, the projective Fra\"{i}ss\'{e} limit of the family $\f$ of finite  fans. 
 This in turn defines a partial order on $L$ by $x\preceq_L y$ if and only if for some (equivalently, for any) $v,w\in\lel$ such that $\pi(v)=x$ and $\pi(w)=y$, 
we have $v\preceq_{\lel} w$, where  $\pi: \lel \to L$ is the quotient map.
Whenever we talk about downwards closed sets on $P\in\f^*$ or on $L$, we will  understand that they are downwards closed with respect to $\preceq_P$ or
$\preceq_L$, respectively.

\begin{rem}
{\rm A closed set in $L$ is downwards closed if and only if it is connected and it contains the root of $L$.}
\end{rem}

\begin{lemma}\label{cm}
Every downwards closed maximal chain on $P\in\f^*$  
is maximal.
\end{lemma}
\begin{proof}
It is not hard to see that the conclusion is true for a structure  in $\f$.
Let $\c$  be a downwards closed maximal chain on $P=\varprojlim(A_n, f^n_m)\in\f^*$. Then for each $n$, $\c^{A_n}=\{ f^{\infty}_n(C): C\in\c\}$ is a downwards closed chain
which is  maximal, by the same argument as in Lemma~\ref{image3}.
If a closed set $D\subset P$ is such that $\c\cup\{D\}$ is a chain, then for each $n$, $f^{\infty}_n(D)\in\c^{A_n}$
by the maximality of $\c^{A_n}$ and therefore $f^{\infty}_n(D)$ is downwards closed, and consequently so is $D$,
which implies $D\in\c$.
\end{proof}

\begin{prop}\label{ccompact}
For every $P\in\f^*$  
the set of all downwards closed maximal chains on $P$ is compact.
In particular, the set of all downwards closed maximal chains on $\lel$ is compact.
\end{prop}
\begin{proof}
We first show that the set  ${\rm CExp}(P)$ of all downwards closed closed subsets of $P=\varprojlim(A_n, f^n_m)$ 
is closed in ${\rm Exp}(P)$. 
Let $K\subset P$ be closed but not downwards closed, witnessed by $x\notin K$ and $y\in K$ be such that $x\preceq_P y$.
Pick $n$ and $a\in A_n$ such that $a=f^\infty_n(x)\neq f^\infty_n(y)$ and $A=(f^\infty_n)^{-1}(a)\cap K=\emptyset$.
Let $B=(f^\infty_n)^{-1}(\{b\in A_n: a\prec_{A_n} b\})$. Clearly $B$ is open and $y\in B$. Then
\[  V := [B,P]\cap [P\setminus A]=\{L\in {\rm Exp}(P): L\cap B\neq\emptyset {\rm \ and \ } L\subset P\setminus A\} \]
is such that $K\in V$ and all sets in $V$ are not downwards closed, which finishes the proof that  ${\rm CExp}(P)$ is closed.

Since, Uspenskij proved in \cite{U} that the set of maximal  chains is closed in ${\rm Exp (Exp}(P))$, by Lemma \ref{cm},
it is enough to show that the set of points in  ${\rm Exp (Exp}(P))$ consisting of sets contained  ${\rm CExp}(P)$ is again a closed sets.
This last thing follows from the following simple general observation: If $K$ is a compact  space and $D\in {\rm Exp}(K)$,
then $\{E\in {\rm Exp}(K): E\subset D\}$ is closed in ${\rm Exp}(K)$. Finally, take $K={\rm Exp} (P)$ and $D={\rm CExp}(P)$.
\end{proof}

\subsection{Precompact uniform spaces}\label{pus} 

A good introduction to   uniform spaces can be found in Engelking \cite{E}, Chapter 8 (precompact spaces  are called totally bounded there),
Below we  briefly review the very minimum that is needed for the paper, all undefined concepts are in Engelking \cite{E}.

A {\em uniformity} is a set $X$ together with a family $\u$ of subsets of $X\times X$ having the following properties:
\begin{enumerate}
\item each $U\in\u$ contains the diagonal $\{(x,x): x\in X\}$;
\item if $U\in \u$ and $U\subset V$, then $V\in\u$;
\item if $U,V\in\u$, then $U\cap V\in\u$;
\item if $U\in\u$, then $U^{-1}=\{(y,x): (x,y)\in U\}\in\u$;
\item if $U\in\u$ then there is $V\in \u$ such that $V\circ V=\{(x,z): { \rm\  there\  exists\  } y\in X { \rm\  such\  that\  } (x,y)\in V { \rm\ and\  } (y,z)\in V\}\subset U$.
\end{enumerate} 
Every uniform space $(X,\u)$ becomes a topological space if we declare $U\subset X$
to be open if and only if for every $x \in U$ there exists an $ V\in\u$ such that $V[x]=\{y\in X: (x,y)\in V\}\subset~U$.
A function $f: (X,\u)\to (Y,\mathcal{V})$ between uniform spaces is called {\em uniformly continuous} if for every
$V\in\mathcal{V}$ there exists $U\in\u$ such that  $f(U)\subset V$.
We say that a uniform space $(X,\mathcal{U})$ is {\em precompact} if for every $U\in\mathcal{U}$
 there are finitely many $x_1,\ldots, x_n\in X$ such that $X=\{x\in X: \exists_i (x,x_i)\in U\}$.
 Equivalently, a uniform space $(X,\mathcal{U})$ is precompact if its completion is compact.
If $(X,\mathcal{U})$ is metrizable by a metric $d$ (i.e. the topology induced by $(X,\mathcal{U})$ is equal to the topology induced by $d$)
then the completion of $(X,\mathcal{U})$ is
equal to the completion of $(X,d)$ (see Lemma 8.3.7 and  Proposition 8.3.5 in Engelking \cite{E}).
This implies, 
$(X,\mathcal{U})$ is precompact if and only if the metric space  $(X,d)$ is precompact. 

Recall  that any compact   space $X$ has the unique uniformity compatible with the topology. This uniformity 
consists of all symmetric neighbourhoods of the diagonal in $X\times X$.

A topological group $G$ 
admits a few natural uniform structures compatible with its topology. We will be working with the right uniformity, which
is generated by the sets
\[ O_V=\{(x,y): xy^{-1}\in V\},\]
where $V$ is an open symmetric neighbourhood of the identity in $G$. 
For a closed subgroup $H$ of $G$ we consider the quotient space $G/H$ with the  quotient uniformity generated by the sets
\[ U_V=\{(xH,yH): xy^{-1}\in V\}=\{(gH,vgH): g\in H, v\in V\}  ,\]
where $V$ is an open symmetric neighbourhood of the identity in $G$. This uniformity is compatible with the quotient topology of $G/H$ and $G/H$ is precompact if and only if
 for every open symmetric neighbourhood $V$ of the identity in $G$, there exist finitely many $x_1,\ldots, x_n\in G$
 such that $G=\bigcup_{i=1}^n (Vx_iH)$. 

If $G$ is a Polish group and $d_R$ is a right-invariant metric on $G$, then the uniformity on $G/H$ is metrizable by the metric
\[ d(g_1H, g_2H)=\inf_{h\in H} d_R(g_1 h, g_2).\]

The following is a folklore, but we could not find a proof, therefore we  include it here.
\begin{prop}\label{extension}
Suppose that $G/H$ is precompact, where $G,H$ are Polish groups, $H$ is a closed subgroup of $G$.
 Then the  continuous action of $G$ on $G/H$ by left translations, $g_1\cdot (g_2H)=(g_1g_2)H$, extends 
to a continuous action of $G$ on the completion $\widehat{G/H}$.
\end{prop} 

Suppose that $G$ acts  on a uniform space $X=(X,\u_X)$ by uniform space isomorphisms. The action is called {\em bounded or motion equicontinuous} 
(see Pestov \cite{P2}, page 70,
and references therein)  if for every $U\in\u_X$, the set
$\{g\in G: \forall x\in X \  (x, g\cdot x)\in U\}$ is a neighbourhood of 1 in $G$.

Assuming additionally that  $X$ is a Polish space and $G$ is a Polish group, we immediately see the following.
\begin{enumerate}
\item If $X$ is a completion of an invariant space $X_0$ then if the action of $G$ on $X_0$ is motion equicontinuous then so is 
the action of $G$ on $X$. 
\item If the action of $G$ on $X$ is motion equicontinuous then for a fixed $x\in X$ the function $g\to g\cdot x$ is continuous 
with respect to the compatible topologies.
(The definition immediately implies that  $g\to g\cdot x$ is continuous at the identity,
which implies that this function is in fact continuous.)
\end{enumerate}

\begin{proof}[Proof of Proposition \ref{extension}]
First observe that for a fixed $g\in G$, the bijection  $f_g: G/H\to G/H$ given by $f_g(hH)=ghH$  is uniformly continuous.
Indeed, for any open symmetric neighbourhood $1\in V$ in $G$, we have $(f_g\times f_g)(U_{g^{-1}Vg})\subset U_V$.
Similarly, $f^{-1}_g$ is uniformly continuous.
Therefore $f_g$ extends to a uniform isomorphism of $\widehat{G/H}$. This gives an action of $G$ on $\widehat{G/H}$, which is continuous if we fix $g\in G$.
Since separately continuous functions are continuous
(see \cite{BKe}, Proposition 2.2.1), it suffices to show that it is continuous if we fix $x\in\widehat{G/H}$. 

By the remarks before the proof, it suffices to check that
for a fixed $hH$
 the function from $G$ to $G/H$, $g\to ghH$ is motion equicontinuous.
However, this is clear, as for any  open symmetric neighhourhood $1\in V$ in $G$, $h\in G$,  and $g\in V$ we have $(hH, ghH)\in \u_V$.
\end{proof}

\subsection{Kechris-Pestov-Todorcevic correspondence}\label{sec:KPT}

In this section, we review  the Kechris-Pestov-Todorcevic correspondence between the structural Ramsey theory of a Fra\"{i}ss\'{e}-HP family and the dynamics (extreme amenability, the universal minimal flow)
of the automorphism group of its Fra\"{i}ss\'{e} limit.

 A topological group $G$ is {\em extremely amenable} if every $G$-flow has a fixed point.
 A~{\em colouring} of a set $X$ is any function $c: X\to \{1,2,\ldots,r\}$, for some $r\geq 2$;
 we say that $Y\subset X$ is  {\em $c$-monochromatic} (or just {\em monochromatic}) if $r\rest Y$ is constant.
 
Let $\g$ be a family of finite structures in a language $\ll$.
For $A,B$ in $\g$,  let ${B \choose A}$~denotes the set of all embeddings of $A$ into $B$.
We say that $\g$ is a {\em Ramsey class} if for every  integer $r\geq 2$ and for
 $A,B\in \g$ there exists $C\in \g$ such that for every colouring $c: {C \choose A} \to\{1,2,\ldots,r\}$ there exists $h\in {C \choose B}$ such that 
$\{ h\circ f: f\in {B \choose A} \}$ is monochromatic. 
We say that $A\in\g$ is {\em rigid} if it has a  trivial automorphism groups. 
 
Kechris-Pestov-Todorcevic \cite{KPT} worked with
 Fra\"{i}ss\'{e} families and their ordered Fra\"{i}ss\'{e} expansions, which was generalized by then Nguyen Van Th\'e \cite{NVT} to 
  Fra\"{i}ss\'{e} families and to arbitrary relational Fra\"{i}ss\'{e} expansions. The Kechris-Pestov-Todorcevic correspondence remains true for Fra\"{i}ss\'{e}-HP families,
  which was checked by several people, and it appears in \cite{Z}.

   \begin{thm}[Kechris-Pestov-Todorcevic \cite{KPT}, see Theorem 5.1 in \cite{Z}] \label{kpt1}
   Let $\g$ be a Fra\"{i}ss\'{e}-HP,  let $\mathbb{G}$ be its Fra\"{i}ss\'{e} limit, and let $G=\aut(\mathbb{G})$.
Then the following are equivalent: 
 \begin{enumerate}
 \item The group $G$ is extremely  amenable.
 \item The family $\g$ is a Ramsey class and it consists of rigid structures.
 \end{enumerate}
  \end{thm}

    Let $\g$ be a Fra\"{i}ss\'{e}-HP family in a language $\ll$,  let $\mathbb{G}$ be its Fra\"{i}ss\'{e} limit, and let $G=\aut(\mathbb{G})$.
    Let $\g^*$ be a  Fra\"{i}ss\'{e}-HP family, in a language $\ll^*\supset \ll$, $\ll^*\setminus \ll$  relational,
     such that for every $A^*\in\g^*$, $A^*\rest\ll\in\g$, that is, every $A^*\in\g^*$ is an {\em expansion} of some $A\in\g$, or $\g^*$ is an {\em expansion} of $\g$.
  Let $\mathbb{G}^*$ be the  Fra\"{i}ss\'{e} limit of $\g$, and let $G^*=\aut(\mathbb{G}^*)$. 
  
We say that the   expansion $\g^*$ of $\g$ is {\em reasonable}, that is, 
  for any $A,B\in\g$, an embedding $\alpha: A\to B$ and an expansion $A^*\in\g^*$ of $A$,
there is an expansion $B^*\in\g^*$ of $B$ such that $\alpha: A^*\to B^*$ is an embedding. It is
   {\em precompact} if  
 for every $A\in\g$ there are only finitely many $A^*\in\g^*$ such that $A^*\restriction \ll= A$. 
  We say that $\g^*$ has the {\em expansion property}  relative to $\g$ if for any $A^*\in\g^*$ there is 
  $B\in\g$ such that for any expansion  $B^*\in\g^*$, there is an embedding $\alpha\colon A^*\to B^*$.

\begin{prop}[\cite{KPT}, \cite{NVT}, see Proposition 5.3 in \cite{Z}]\label{kpt_reas}
The expansion $\g^*$ of $\g$ is  reasonable if and only if $\mathbb{G}^*\rest \ll= \mathbb{G}$.
\end{prop}

{\bf{From now on till the end of this section,}} we will assume that the expansion $\g^*$ of $\g$ is  reasonable, precompact, and satisfies the property $(*)$ below.
\begin{equation*}
\begin{split}
(*) & \text{ For any $A\in\g$ and an embedding $i:A\to \mathbb{G}$ there is an expansion }\\ 
 & \text{$A^*\in\g^*$ of $A$ such that $i:A^*\to \mathbb{G}^*$ is an embedding.}
\end{split}
\end{equation*}

   Below  $(\mathbb{G}, \vec{R})$, $(\mathbb{G}, \vec{S})$, etc.  denote an expansion of  $\mathbb{G}$ to a structure in $\ll^*$.
 Instead of $(\mathbb{G}, \vec{R})$ we will often  just write $\vec{R}$.   
   
Define 
\begin{equation*}
\begin{split}
X_{\mathbb{G}^*}=&\{ \vec{R}:
{\rm\ for\ every\ } A\in\g, 
{\rm\ and\ an\ embedding\ }
i: A\to \mathbb{G} {\rm\ there \ exists\ } \\ 
& A^*\in\g^*, {\rm such \ that\ } i: A^*\to(\mathbb{G},\vec{R})
{\rm \ is\ an\ embedding }\}.
\end{split}
\end{equation*}

We make $X_{\mathbb{G}^*}$ a topological space by declaring sets
\[ V_{i, A^*}=\{\vec{R}\in X_{\mathbb{G}^*} : i: A^*\to (\mathbb{G},\vec{R})  \text{  is an embedding} \}, \]
where $i: A\to \mathbb{G}$ is an embedding,  $A^*\in\g^*$, and $A^*\rest \ll= A$, to be open.
The group $\aut(\G^*)$ acts continuously on $X_{\mathbb{G}^*}$ via
\[ g\cdot \vec{R}(\bar{a})=\vec{R}(g^{-1}(\bar{a})).\]
Reasonability and  precompactness  of the expansion $\g^*$ of $\g$ imply that the space $X_{\mathbb{G}^*}$ is compact,
zero-dimensional, and it is  nonempty as $\mathbb{G}^*\in X_{\mathbb{G}^*}$.

\begin{thm}[\cite{KPT}, \cite{NVT}, see Proposition 5.5 in \cite{Z}]\label{kpt_minim}
 The following are equivalent:
 \begin{enumerate}
 \item The flow $G\acts X_{\mathbb{G}^*} $ is minimal.
 \item  The  family $\g^*$ has  the expansion property relative to $\g$.
 \end{enumerate}
\end{thm}
   
  \begin{thm}[Kechris-Pestov-Todorcevic \cite{KPT}, Nguyen Van Th\'e \cite{NVT}, see Theorem 5.7 in \cite{Z}]\label{kpt2}
The following are equivalent:
 \begin{enumerate}
 \item The flow $G\acts X_{\mathbb{G}^*} $ is the universal minimal flow of $G$.
 \item The  family $\g^*$ is a rigid Ramsey class and  has the expansion property relative to $\g$.
 \end{enumerate}
  \end{thm}

 We finish this section with  several general observations,  which are adaptations of those in \cite{NVT} (pages 6-8) to the framework of the Fra\"{i}ss\'{e}-HP theory.

For an embedding $\alpha: A\to \G$, $A\in\g$, let $V_\alpha$ denote the pointwise stabilizer of $\alpha$, that is,
\[
 V_\alpha=\{g\in \aut(\G): \text{ for every } a\in A, \ g(\alpha(a))=\alpha(a) \}.
 \]
Then $V_{\alpha}$ is a symmetric clopen neighbourhood of the identity in $\aut(\G)$, in fact it is also a subgroup of $\aut(\G)$,  and
sets of this form constitute a neighbourhood basis of the identity in $\aut(\G)$.

\begin{lemma}
The right uniform space $\aut(\G)/\aut(\G^*)$ is precompact.
\end{lemma} 
\begin{proof}
Let $V=V_{\alpha}$ for some embedding $\alpha: A\to\G$.  Enumerate all expansions of $A$  in $\mathbb{G}^*$ as
$A^*_1, A^*_2,\ldots, A^*_N$. For each $i=1,2,\ldots , N,$ using the  universality of $\G^*$,
 pick an embedding $y_i: A^*_i\to\G^*$.
The ultrahomogeneity of $\G$ with respect to $\g$ implies that
there  are $x_i\in \aut(\G)$ such that $y_i=x_i\circ\alpha:   A^*_i\to\G^*.$ 
Pick any $g\in \aut(\G)$ and we will show that $g^{-1}\in Vx_i^{-1}\aut(\G^*)$, for some $i=1,2,\ldots , N$, which will finish the proof of the lemma.
Using the property $(*)$ , take $i$  such that $g\circ\alpha: A^*_i\to\G^*$ is an embedding. From the  ultrahomogeneity of $\G^*$,
we get $h\in \aut(\G^*)$ such that $h\circ g\circ\alpha=x_i\circ\alpha$. That implies $x_i^{-1}\circ h\circ g\in V$,  and hence we get
$g^{-1}\in Vx_i^{-1}\aut(\G^*)$.
\end{proof}


As we have just seen, the space  $X_{\mathbb{G}^*}$ is compact and that  the right uniform space $\aut(\G)/\aut(\G^*)$ is precompact,
we show in Theorem \ref{iden} that $X_{\mathbb{G}^*}$ and $\widehat{\aut(\G)/\aut(\G^*)}$ are isomorphic.
Call a map $t:X\to Y$, such that $X,Y$ are uniform spaces and a topological group $G$ acts continuously on both $X$ and $Y$  a {\em  uniform $G$-isomorphism}
if it is a $G$-map which is an isomorphism between the uniform spaces $X$ and $Y$.
  \begin{thm}\label{iden}
  The map $g\aut(\G^*)\to g\cdot \vec{R}^\mathbb{G}$ from $\aut(\G)/\aut(\G^*)$ to $X_{\mathbb{G}^*}$ is a uniform $G$-isomorphism
from $G\acts \aut(\G)/\aut(\G^*)$ to $G\acts X_{\mathbb{G}^*} $.
  \end{thm}
  
We will say that flows $G\acts X$ and $G\acts Y$ are {\em isomorphic} if there is a homeomorphism from $X$ onto $Y$ which is a $G$-map.
  \begin{cor}\label{ident}
  The flow $G\acts \widehat{\aut(\G)/\aut(\G^*)}$ is isomorphic to the flow $G\acts X_{\mathbb{G}^*} $.
  \end{cor}
  

To prove Theorem \ref{iden}, first we show the following lemma.
   \begin{lemma}
\begin{enumerate}
\item The uniformity on $X_{\mathbb{G}^*}$  is generated by sets
\begin{align*}
 U^\alpha=&\{(\vec{R}, \vec{S}): \text{ for some } A^*\in\g^*\text{ with } A^*\rest \ll= A\   \\
& \alpha: A^*\to(\G, \vec{R})\text{ and } \alpha: A^*\to (\G, \vec{S})\text{ are embeddings}\},
\end{align*}
where $\alpha: A\to\G$ is an embedding.
\item The quotient uniformity on $\aut(\G)/\aut(\G^*)$ 
given by the right uniformity on $\aut(\G)$
is generated by sets
\[U_\alpha=\{(x\aut(\G^*),y\aut(\G^*)): x^{-1}\circ \alpha=y^{-1} \circ \alpha\},\]
where $\alpha: A\to \G$ is an embedding. 
\end{enumerate}

\end{lemma}

\begin{proof}
To see (1), using the compactness of $X_{\mathbb{G}^*}$, note that for each open neighbourhood $U$ of the diagonal  of $X_{\mathbb{G}^*}$
 we can find $A\in\mathcal{G}$ and an embedding $\alpha: A\to \G$ such that the partition
of  $X_{\mathbb{G}^*}$ into clopen sets 
\[ V_{\alpha, A^*}=\{\vec{R}\in X_{\mathbb{G}^*} : \alpha: A^*\to(\G, \vec{R})  \text{  is an embedding} \}, \]
where   $A^*\in \ll^*$ and $A^*\rest\ll= A$, has the property that 
\[ \bigcup_{A^*\in\G^*, \ A^*\rest\ll= A}  V_{\alpha, A^*}\times V_{\alpha, A^*}\subset U.\]  

Part (2) follows immediately from the definition of the uniformity on $\aut(\G)/\aut(\G^*)$.
\end{proof}

\begin{proof}[Proof of Theorem \ref{iden}]
Let $\alpha: A\to \G$ be an embedding. It is enough to show that   
 \[U_\alpha=\{(x\aut(\G^*),y\aut(\G^*)): x^{-1}\circ \alpha=y^{-1} \circ \alpha\}\] 
on $\aut(\G)/\aut(\G^*)$ is mapped to 
\begin{align*}
U^\alpha=&\{(x\cdot \vec{R}^\G, y\cdot \vec{R}^\G): \text{ for some } A^*\in\g^*\text{ with } A^*\rest\ll= A\   \\
& \alpha: A^*\to(\G, x\cdot \vec{R}^\G)\text{ and } \alpha: A^*\to (\G, y\cdot \vec{R}^\G)\text{ are embeddings}\} 
\end{align*} 
on   $\aut(\G)\cdot \vec{R}^\G$. 
Clearly, the set $U_\alpha$ is mapped to
\[\overline{U}_\alpha = \{(x\cdot \vec{R}^\G,y\cdot \vec{R}^\G):  x^{-1}\circ \alpha=y^{-1} \circ \alpha\}.\]

Let $(x\cdot \vec{R}^\G, y\cdot \vec{R}^\G)\in U^\alpha$. Let $A^*$ be such that 
$\alpha: A^*\to(\G, x\cdot \vec{R}^\G)$ and $\alpha: A^*\to (\G, y\cdot \vec{R}^\G)$ are embeddings.
Then $x^{-1}\circ\alpha: A^*\to (\G,  \vec{R}^\G)$ and $y^{-1}\circ\alpha: A^*\to (\G,  \vec{R}^\G)$ are embeddings.
By the projective ultrahomogeneity property for $\mathcal{G}^*$, there is $h\in\aut(\G^*)$ such that 
$h \circ x^{-1}\circ\alpha= y^{-1}\circ\alpha$. This implies $((x\circ  h^{-1}) \cdot \vec{R}^\G, y \cdot\vec{R}^\G)\in \overline{U}_\alpha$, and since 
$h^{-1}\in\aut(\G^*)$ and so $h^{-1} \cdot \vec{R}^\G=\vec{R}^\G$, we get  $(x \cdot \vec{R}^\G, y \cdot \vec{R}^\G)\in \overline{U}_\alpha$.

Now suppose that  $(x \cdot \vec{R}^\G, y \cdot \vec{R}^\G)\in \overline{U}_\alpha$.
From the property $(*)$, it follows that there is $A^*\in\g^*$ with $A^*\rest\ll= A$ such that $x^{-1}\circ \alpha : A^*\to (\G,  \vec{R}^\G)$ is an embedding.
Since $x^{-1}\circ\alpha=y^{-1}\circ\alpha$, clearly $y^{-1}\circ \alpha : A^*\to (\G,  \vec{R}^\G)$ is an embedding.
This implies that $\alpha : A^*\to(\G, x\cdot \vec{R}^\G)$ and $\alpha :A^*\to (\G, y\cdot \vec{R}^\G)$ are embeddings,
which gives  $(x\cdot \vec{R}^\G, y\cdot \vec{R}^\G)\in U^\alpha$. 

\end{proof}

\section{The universal minimal flow of $\aut(\lel)$}



We will expand each finite fan in $\f$ by a maximal chain of downwards closed subsets and obtain a class $\f_c,$ which does not directly fall into the framework of the projective Fra\"iss\'e theory. However, we will show that $\f_c$ is equivalent to a Fra\"iss\'e-HP class of first-order structures, which is  reasonable and precompact with respect to $\f.$

In Section 4.2, we prove the main combinatorial result  that $\f_c$ is a Ramsey class. We do so indirectly by showing that a certain class $\f_{cc}$ coinitial in $\f_c$ is Ramsey. In our proof we will use the dual Ramsey theorem of Graham and Rothschild.

Finally, in Section 4.3, we apply methods from Section 3.6 to compute the universal minimal flow of $\aut(\lel)$ in two ways - as a completion of a precompact space equal to the quotient of  $\aut(\lel)$ by an extremely amenable subgroup, and as the space of maximal downwards closed chains on $\lel$ - and we exhibit an explicit isomorphism between them.

\subsection{Finite fans equipped with chains -- the family $\boldsymbol{\f_c}$}\label{sec:fc}


Define
\[\f_c=\{(A, \c^A): A\in\f \text{ and }  \c^A \text{ is a downwards closed maximal chain}\}\]
and for $(A, \c^A), (B, \c^B)\in\f_c$ we say that $f: (B, \c^B)\to (A, \c^A)$ is an {\em epimorphism}
 iff $f: B\to A$ is an epimorphism and
for every $C\in \c^B$, we have $f(C)\in \c^A$ (short: $f(\c^B)=\c^A$).
If  $A\in\f$ and $\c^A$ are such that $A_c=(A,\c^A)\in\f_c$, then $\c^A$ induces a linear order $\leq^{A_c}$ on $A$ given by $x<^{A_c} y$ iff for some $C\in\c^A$, $x\in C$ and $y\notin C$.
For $(A, \c^A)\in\f_c$ we say that  $\c^A$ is {\em canonical} on $A$ if for 
some ordering of branches $b_1<\ldots< b_n$ of $A$, it holds that whenever $C\in\c^A$ and $C\cap b_j\neq\emptyset$, then $b_i\subset C$ for every $1\leq i<j\leq n.$ Analogously as for topological $\ll$-structures, we define the JPP and the AP for the family $\f_c$ with the epimorphisms as above.

\begin{lemma}\label{fcapp}
The family $\f_c$ has the JPP and the AP.
\end{lemma}

\begin{proof} 
As there is $A\in\f_c$ which has just one element,  the AP will imply the JPP. 

For the AP, 
let $f: B\to A$ and $g: D\to A$
be epimorphisms. We find $E$,  $k: E\to B$ and $l: E\to D$
such that $f\circ k=g\circ l$. 

It is straightforward to see that the conclusion holds when each of $A$, $B$, $D$ has exactly one branch.

In the general situation, we take $E$ that has 
the width  $w=b+d-a$, where $b,d,a$ are numbers of elements in $B,D,A,$ respectively, and it has all branches of the height $h$, where 
 $h$ is the height of any $E_0\in \f_c$ that witnesses the AP for $A_0,B_0, D_0\in \f_c$ all with one branch only whose heights are the same as of $A,B,D,$ respectively, and any epimorphisms $f_0:B_0\to A_0$ and $g_0:D_0\to A_0.$ 
Let $\c^E$ be the canonical chain on $E$ with respect to some enumeration $e_1,\ldots, e_w$ of branches of $E$.

Let $(a^i)_{1\leq i \leq a}$ be the increasing enumeration of $A$ according to the linear order $\leq^A$ induced by $\c^A$. 
For every $i$
let $b^i\in B$ and $d^i\in D$ be the the smallest (with respect to the linear orders $\leq^B$ induced by $\c^B$ and  $\leq^D$ induced by $\c^D$, respectively), such that $f(b^i)=a^i$ and $g(d^i)=a^i$.
Let $v_A,v_B,v_D$ be the roots of $A,B,D$, respectively.

We now proceed to define epimorphisms $k: E\to B$ and $l: E\to D$ such that $f\circ k=g\circ l$.
For every $i=1,2,\ldots, a$, 
let $J_i=|\{b\in B: b<^B b^i \}|+|\{d\in D: d<^D d^i\}|-(i-1)$.
For $j=J_i+1$, let $k\rest e_j$ and $l\rest e_j$ be chosen so that $(k\rest e_j) (e_j)=[v_B,b^i]_{\preceq_B}$, 
$(l\rest e_j) (e_j)=[v_D,d^i]_{\preceq_D}$, and $(f\circ k)\rest e_j=(g\circ l)\rest e_j$,
which we can do by the choice of $h$.
Let $b^{i,1},\ldots, b^{i,m_i}$ be the increasing (with respect to $\leq^B$) enumeration of $\{b\in B: b^i<^B b<^B b^{i+1} \}$ and let 
$d^{i,1},\ldots, d^{i,n_i}$ be the increasing (with respect to $\leq^D$) enumeration of $\{d\in D: d^i<^D d<^D d^{i+1} \}$. 
For $j=J_i+m+1$,  $m=1,\ldots, m_i$,
let $k\rest e_j$ and $l\rest e_j$ be such that $(k\rest e_j) (e_j)=[v_B,b^{i,m}]_{\preceq_B}$, 
$(l\rest e_j) (e_j)\leq^D d^i$, and $(f\circ k)\rest e_j=(g\circ l)\rest e_j$.
For $j=J_i+m_i+n+1$,  $n=1,\ldots, n_i$,
let $k\rest e_j$ and $l\rest e_j$ be such that $(k\rest e_j) (e_j)\leq^B b^i$,
$(l\rest e_j) (e_j)=[v_D,d^{i,n}]_{\preceq_D}$,  and $(f\circ k)\rest e_j=(g\circ l)\rest e_j$.
This defines the required epimorphisms $k$ and $l$.

\end{proof}

At this point we would like to say that there is a 
limit of the family $\f_c$, i.e. there is some $\mathbb{F}_c$ which satisfies properties (L1), (L2), and (L3).
However, structures in $\f_c$ are not first order structures, chains are neither functions nor relations, therefore we cannot directly apply the Irwin-Solecki
theorem about the existence and uniqueness of projective Fra\"{i}ss\'{e} limits. 

It does not seem that we can realize the family $\f_c$ as a projective Fra\"{i}ss\'{e} family. In particular, the following natural attempt fails.
\begin{rem}
{\rm 
To a structure $(A,\c^A)\in\f_c$ we  associate a structure $(A,<^A)$, where  $<^A$ is the linear order induced by $\c^A$, 
and we  consider a family $\f_<$ of all $(A,<^A)$ obtained in this way.
However, epimorphisms between structures in $\f_c$ and epimorphisms between structures in $\f_<$ are not the same.
For example, let $A=\{a_1,a_2\}$ consist of a single branch and let $\c^A=\{\{a_1\}, \{a_1, a_2\}\}$,  let $B=\{b_1,b_2,b_3\}$ consist of a single
 branch and let
$\c^B=\{\{b_1\},\{b_1,b_2\}, \{b_1,b_2,b_3\}\}$. Let $\phi$ satisfy $\phi(b_1)=\phi(b_3)=a_1$ and $\phi(b_2)=a_2$.
Then $\phi:(B,\c^B)\to (A,\c^A)$ is an epimorphism, whereas  $\phi:(B,<^B)\to (A,<^A)$ is not an epimorphism.}
\end{rem}

Nevertheless, we will show that the family $\f_c$ can be identified with a Fra\"{i}ss\'{e}-HP family similarly as in Theorem \ref{dual}. For this  we will have to consider
inverse limits of structures in $\f_c$.

Recall from Section \ref{sec:chain} that by $\f^*$ we denoted the family of all topological $\ll$-structures that are countable inverse limits of finite  fans in $\f$. 
If $P\in\f^*$ is the inverse limit of an inverse sequence $(A_n, f_m^n)$, denoted by $P=\varprojlim(A_n, f^n_m),$ we consider the partial order on $P$
given by
\[x\preceq_P y  \text{ iff  for every } n, \  f^\infty_n(x)\preceq_{A_n} f^\infty_n(y),\]
where $\preceq_{A_n}$ is the tree partial order on $A_n$. Downwards closed sets on $P\in\f^*$  will be taken with respect to $\preceq_P.$ 
If   $((A_n,\c^{A_n}), f_m^n)$ is an inverse sequence in $\f_c$, we say that $(P,\c^P)$ is its inverse limit if $P=\varprojlim(A_n, f^n_m)$ and $\c^P$ is the collection of all closed subsets $C$ of $P$ such that $f^\infty_n(C)\in \c^{A_n}$ for every $n$.
We will show that $\c^P$ is a downwards closed maximal chain.
To see that $\c^P$ is a chain, note that
if $C_1, C_2\in\c^P$, then either for all $n$, $f_n^\infty(C_1)\subset  f_n^\infty(C_2)$ or for all $n$,
$f_n^\infty(C_2)\subset  f_n^\infty(C_1)$. In the first case, $C_1\subset C_2$, and in the second, $C_2\subset C_1$.
The chain $\c^P$ is downwards closed. Finally,
the chain $\c^P$ is maximal. Indeed, if a 
downward closed set $C$ is such that $\{C\}\cup \c^P$ is a chain,
then for each $n$, by the maximality of $\c^{A_n}$, $f^{\infty}_n(C)\in \c^{A_n}$, which by the definition of $\c^P$
gives $C\in\c^P$. 

Let $\f_c^*$ be the family of all inverse limits of structures from $\f_c$.
Clearly, we can identify $\f_c$ with a subfamily of $\f^*_c$
by assigning to $(A,\c^A)$ the inverse limit of $((A,\c^{A}), \Id_m^n)$.
 For $A_c=(A,\c^A)\in\f^*_c$ denote by $A_c\rest\ll$ 
the structure $A$.
Generalizing the definition for $\f_c$, for $ (Q, \c^Q), (P, \c^P)\in\f^*_c$ we will say that $f: (Q, \c^Q)\to (P, \c^P)$ is an {\em epimorphism}
 iff $f: Q\to P$ is an epimorphism and for every $C\in \c^Q$, we have $f(C)\in \c^P$. If $(P, \c^P)=\varprojlim((A_n,\c^{A_n}), f^n_m)$, then $f^\infty_n: (P,\c^P)\to (A_n, \c^{A_n})$ is an epimorphism for every~$n$.

We say that a function $f:(Q,\c^Q)\to (P,\c^P)$ is chain preserving iff $f(\c^Q)= \c^P$.

\begin{thm}\label{stone2}
The family $\f_c^*$ with epimorphisms is equivalent via a contravariant functor to a family of first order structures with embeddings.
\end{thm}
\begin{proof}

Let $R$ be a binary relation symbol and
take the language $\{S,\leq_{BA},\cup,\cap,^-,0,1\}$, where $S$  and $\leq_{BA}$ are binary relation symbols
and $\{\cup,\cap,^-, 0,1\}$ is the language of Boolean algebras.
For $K=(K,R^K, \c^K)\in\f^*_c$, let $M=(M, S^M,\leq_{BA}^M, \cup^M, \cap^M, ^{-M}, 0^M, 1^M\}$
be the structure
such that $M=\Clop(K)$ is  the family of all clopen sets of $K$,  $\cup^M$ is the union, $\cap^M$ is the intersection,  $^{-M}$ is the complement,
$0^M$ is the empty set and $1^M=M$.
As in Proposition \ref{stone1}, we set  for every   $X, Y\in M$, $S^M(X, Y)$ if and only if
   for some $a\in X, b\in Y$,
  we have $R^K(a,b)$. 

We first define $\leq_{BA}$ for $K\in\f_c$, and then we provide a definition for $K\in\f^*_c$. Let then  $K\in\f_c$.
As before, denote by $\leq^K$  the linear order on $K$ induced by $\c^K$ by letting $x<^K y$ iff there exists $C\in\c^K$ such that $x\in C$ and $y\notin C$.
From the maximality of $\c^K$, the order $\leq^K$ is total. 
For $K\in\f_c$, let $\leq^K_{op}$ denote the order opposite to the order $\leq^K$, that is, we let $x\leq^K_{op} y$ iff $y\leq^K  x$.

Take $\leq_{BA}^M $ to be the antilexicographical order with respect to $\leq^K_{op} $, that is for $X,Y\in M=\Clop(K)=P(K)$, where $P(K)$ denotes the power set
of $K$, let
 $X <_{BA}^M Y$  iff
for $a\in K$  which is the largest with respect to $\leq^K_{op} $ such that $a\in X\triangle Y$, we have $a\in Y$. 

Let $f:L\to K$, where $K,L\in\f$ be a continuous surjection and let $F: P(K)\to P(L)$  be the  map given by $F(X)=f^{-1}(X)$.

\smallskip

 \noindent {\bf{Claim.}}
 $F$ is $\leq_{BA}$-preserving iff  $f(\c^L)= \c^K$.   
 \begin{proof}
 The function $f$ is chains preserving iff $f$ maps cofinal segments in $\leq^L_{op}$ to cofinal segments in $\leq^K_{op}$ , i.e. $f$ maps sets
$\{z\in L: a \leq^L_{op} z\}$, some $a\in L$ to $\{z\in K: b \leq^K_{op} z\}$, some $b\in K$ iff
 $F$ is $\leq_{BA}$-preserving.
  \end{proof}
 Proposition \ref{stone1} together with the claim above, already imply the conclusion of the theorem for the family $\f_c$.
 
 Let $\g$ be the family of all $M$'s obtained in this way from some $K\in\f_c$. Maps 
we consider between structures in $\g$ are 
embeddings.

Now let us come to the general situation where the structures come from $\f^*_c$. The claim above implies that an inverse sequence
$((K_n,\c^{K_n}), f_m^n)$  in $\f_c$
corresponds to a direct sequence $((M_n,\leq_{BA}^{M_n}), g_m^n)$  in $\g$. 
Let $(K, \c^K)$ together with epimorphisms $f^\infty_n: (K,\c^K)\to (K_n, \c^{K_n})$ be the inverse limit of $((K_n,\c^{K_n}), f_m^n)$.
Let $M$ together with embeddings $g^\infty_n: M_n\to M$ be the direct limit of $(M_n, g_m^n)$. 

By the definition of the direct limit, $X \leq_{BA}^M Y$ iff  for some (equivalently every) $n$ such that there are  $X_n, Y_n\in M_n$ with $g_n^\infty(X_n)=X$
and $g_n^\infty(Y_n)=Y$, we have   $ X_n \leq_{BA}^{M_n} Y_n $.

Then  $(M,\leq_{BA}^M)$ together with embeddings $g^\infty_n: M_n\to M$ is the direct limit of $((M_n,\leq_{BA}^{M_n}), g_m^n)$.

The following claim will finish the proof.

\noindent{\bf Claim.}
 Let $f:L\to K$, where $K,L\in\f^*_c$, be a continuous surjection  and let $F: \Clop(K)\to \Clop(L)$  be the  map given by $F(X)=f^{-1}(X)$.
 Then $f$ preserves chains iff $F$ preserves $\leq_{BA}$.
 \begin{proof}
 Let $(L_n, l_m^n)$ be an inverse sequence in $\f_c$ with the limit $L$ and
  let $(K_n, k_m^n)$ be an inverse sequence in $\f_c$ with the limit $K$.
  Let $p_m^n: \Clop(K_m) \to \Clop(K_n)$ be the dual map to $k_m^n$ and let 
  $q_m^n: \Clop(L_m) \to \Clop(L_n)$ be the dual map to $l_m^n$. Then 
  $(\Clop(K_m), p_m^n)$ is an inverse sequence with the limit $\Clop(K)$ and  
   $(\Clop(L_m), q_m^n)$ is an inverse sequence with the limit $\Clop(L)$.

Having $f: L\to K$ find a strictly  increasing sequence     $(t_n)$ and continuous surjections 
$h_n: L_{t_{n}}\to K_n$ 
 such that 
 $h_n\circ l^{t_{n+1}}_{t_n}=k^{n+1}_{n}\circ h_{n+1}$.
Let 
 $H_n: \Clop( K_{n}) \to \Clop(L_{t_{n}})$
 be dual maps to $h_n$.

Then we have: $f: L\to K$ preserves chains iff for every $n$,  $h_n$ preserve chains iff for every $n$,  $H_n$ preserve linear orders
iff $F: \Clop(K)\to \Clop(L)$ preserves linear orders.

 \end{proof}

\end{proof}

Thanks to Theorem \ref{stone2}, we  not only know that there is a  unique up to an isomorphism structure  in $\f^*_c$ that satisfies conditions (L1), (L2) and (L3)
for the family $\f_c$,
but also all theorems that were proved for Fra\"{i}ss\'{e}-HP families are available to us.

Let $\lel_c$ denote the limit of the family $\f_c$.
The proposition below tells us that $\lel_c$ is equal to $(\lel, \c^\lel)$ for some downwards closed maximal chain $\c^\lel$.

\begin{lemma}\label{reason}
The expansion $\f_c$ of $\f$ is reasonable, that is, for every $A,B\in\f$, an epimorphism $\phi: B\to A$,  and $A_c\in\f_c$
such that $A_c\restriction\ll= A$, there is $B_c\in\f_c$ such that $B_c\restriction\ll= B$ and $\phi: B_c\to A_c$
is an epimorphism. 
\end{lemma}

\begin{proof}
Let $A,B\in\f$, an epimorphism $\phi: B\to A$,  and $A_c\in\f_c$
such that $A_c\restriction\ll= A$, be given.
We get $\c^B$ by extending in an arbitrary way the downwards closed chain $\{\phi^{-1}(C): C\in\c^A\}$ into a downwards closed maximal chain.
\end{proof}

Lemma \ref{reason} and Proposition \ref{kpt_reas} immediately imply the following corollary.

\begin{cor}\label{reasoncor} 
We have $\lel_c\rest \ll= \lel$.
\end{cor}

Corollary \ref{reasoncor}  implies that $\aut(\lel_c)$, the automorphism group of $\lel_c$, is a subgroup of $\aut(\lel)$,
the automorphism group of $\lel$. Recall that the group $\aut(\lel)$ is equipped with the compact open topology inherited from $H(\lel)$, the homeomorphism group of $\lel$.


\begin{lemma}\label{autclosed}
The group $\aut(\lel_c)$ is a closed subgroup of $\aut(\lel)$.  
\end{lemma}
\begin{proof}
For any $D\in {\rm Exp}(\lel)$,  the map that takes $f\in \aut(\lel)$
and assigns to it $f(D)\in {\rm Exp}(\lel)$ is continuous.
Since $\c^\lel$ is maximal,  it is closed in $\Exp(\lel)$, which implies that  $\aut(\lel_c)$ is closed in $\aut(\lel)$.
\end{proof}
Lemma \ref{autclosed}  also follows from Proposition \ref{cont}, which we prove later.

\subsection{ Extreme amenability of $\boldsymbol{\aut(\lel_c)}$}\label{extamen}
In this section, we define a family $\f_{cc}$ coinitial in $\f_c$ and show that it is a Ramsey class.
This will imply that   $\aut(\lel_c)$ is extremely amenable and that $\f_c$ is a Ramsey class.

Recall that for $A_c\in\f_c$, $\c^A$ is canonical 
if there is an order on branches of $A_c$, which we denote by $\leq_{cc}^{A_c}$, given by $b\leq_{cc}^{A_c} c$ iff for every 
$x\in b$ and $y\in c$, we have $x\leq^{A_c}_c y$.
Let
\[\f_{cc}=\{(A, \c^A)\in\f_c: \c^A \text{ is canonical and all branches in }  A \text{ have the same height}   \}.\]

For $A,B\in\f_{cc}$,  let ${B \choose A}$~denote the set of all epimorphisms from $B$ onto $A$.

The main result of this section is the following theorem. 

\begin{thm}\label{rams}
The class $\f_{cc}$ is a Ramsey class, that is,  for every  integer $r\geq 2$ and for
 $S,T\in \f_{cc}$ with ${T \choose S}\neq\emptyset$
 there exists $U\in \f_{cc}$ such that for every colouring $e: {U \choose S} \to\{1,2,\ldots,r\}$ there is $g\in {U \choose T}$ such that 
$\{ h\circ g: h\in {T \choose S} \}$ is monochromatic.
\end{thm}

\begin{prop}\label{coin}
The family $\f_{cc}$ is  coinitial in $\f_c$, that is, for every $A_c\in\f_c$ there exist $B_c\in\f_{cc}$ and an epimorphism from $B_c$ onto $A_c$.
Moreover, we can choose $B_c$ in a way that its height and width depend only on the height and width of $A_c$.
\end{prop}
\begin{proof}
Let $A_c\in\f_c$  be of  height $k$ and let $v_{A_c}$ denote its root.
Let $B_c\in\f_{cc}$ be of  height $k$ and  width $l$ equal to the number of elements in $A_c$, and let $\c^{B_c}$ be canonical.
Enumerate $A_c$ according to $\leq^{A_c}$ into $a^1,\ldots a^l$, and enumerate branches in $B_c$ according to $\leq^{B_c}_{cc}$ into $b_1,\ldots, b_l$.
Now let for each $i=1,\ldots l$, the branch $b_i$ be mapped onto the segment $[v_{A_c}, a^i]_{\preceq_{A}}$ in $A_c$ in an $R$-preserving way.
This defines a required epimorphism from $B_c$ onto~$A_c$.
\end{proof}

\begin{rem}\label{fccfc}
{\rm Proposition \ref{coin}  implies 
(by Lemma \ref{fcapp}  and Remark \ref{cofinall}) that $\f_{cc}$ 
satisfies the JPP and the AP and that the limits of $\f_c$ and $\f_{cc}$ are isomorphic to $\lel_c$.}  
\end{rem}

From Theorems \ref{kpt1} and \ref{rams}, using  Remark \ref{fccfc}, we will obtain the following corollary. 
\begin{cor}
The automorphism group $\aut(\lel_c)$ is extremely amenable.
\end{cor}

The family $\f_{cc}$ is easier to work with than the family $\f_c$.
Nevertheless, $\f_c$ is a Ramsey class as well, which follows from the following proposition.

\begin{prop}
Let $\g_1\subset \g_2$ be Fra\"{i}ss\'{e}-HP families and suppose that $\g_1$ is coinitial in $\g_2$. 
If $\g_1$ is a Ramsey class, so is $\g_2$.
\end{prop}
\begin{proof}
Let $\mathbb{G}$ be the Fra\"{i}ss\'{e} limit of both $\g_1$ and $\g_2$ and let $G={\rm Aut}(\mathbb{G})$.
As $\g_1$ is a Ramsey class, Theorem \ref{kpt1} (applied to $G$ and $\g_1$) implies that $G$ is extremely amenable. 
Then again applying Theorem \ref{kpt1}, this time to $G$ and $\g_2$, we get that $\g_2$ is a Ramsey class.
\end{proof}

\begin{cor}\label{xyza}
The family $\f_c$ is a Ramsey class.

\end{cor}

The main two ingredients in the proof of Theorem \ref{rams} will be Theorem \ref{lelek} and 
 Corollary \ref{dhe}.

Let $\mathbb{N}=\{1,2,3,\ldots\}$ denote the set of natural numbers and let $k\in\n$. For a function 
$p:\mathbb{N}\to\{0,1,\ldots,k\}$, we define the {\em support } of $p$ to be the set
$\supp(p)=\{l\in\mathbb{N}: p(l)\neq 0\}$.
Let  \[
\fin_k=\{p:\mathbb{N}\to\{0,1,\ldots,k\}: \supp(p) {\rm\ is\ finite\ and\  } \exists l\in\supp(p)\  \left(p(l)=k\right)\}, \]
 and for each $n\in\mathbb{N}$, let
\[\fin_k(n)=\{p\in\fin_k: \supp(p)\subset\{1,2,\ldots,n\}\}.\]
We equip $\fin_k$  and each $\fin_k(n)$ with a partial semigroup operation $+$ defined 
	for $p$ and $q$ whenever $\supp(p)\cap \supp(q)=\emptyset$ by  $(p+q)(x)= p(x)+ q(x)$. 

The  Gowers'  {\em tetris} operation is the function $T: \fin_k\to \fin_{k-1}$ defined by
\[  T(p)(l)=\max\{0,p(l)-1\}.\]
We define for every $0<i\leq k$ 
 a function  $T^{(k)}_i:\fin_k\to\fin_{k-1}$, 
which behaves like the identity up to the value $i-1$ and like the tetris above it  as follows.
\[
		T^{(k)}_i(p)(l) =
\begin{cases}
   p(l) &\text{if } p(l)<i \\
   p(l)-1  &\text{if } p(l)\geq i.
\end{cases}
 \]
 We also define $T^{(k)}_0=\Id\restriction_{\fin_k}$. 
 It may seem more natural to denote the identity
by $T^{(k)}_{k+1}$ or $T^{(k)}_\infty$, only for notational convenience later on, we will be using
 $T^{(k)}_0$. 
Note that  $T^{(k)}_1$ is the Gowers' tetris operation.
 We will usually drop superscripts and write $T_i$ rather than $T^{(k)}_i$.

Define
\begin{equation*}
\begin{split}
\fin_k^{(d)}(n)=&\{(p_1,\ldots, p_d) : p_i\in\fin_k(n) \text{ and } 
\forall_{i< j} \left(\supp(p_i)\cap\supp(p_j)=\emptyset\text{ and } \right. \\ 
&\left. \min(\supp(p_i))<\min(\supp(p_j))\right) \} 
\end{split}
\end{equation*} and
\[\fin_k^{*(d)}(n)=\{(p_1,\ldots, p_d) \in\fin_k^{(d)} :  
\min(\supp_k(p_{i}))<\min(\supp(p_{i+1}))  \},\]
where 
$\supp_j(p)=\{l\in\{1,\ldots, n\}: p(l)=j\}$ for $p\in\fin_k(n)$ and $j=1,\ldots, k.$

Let $P_k=\prod_{j=1}^k\{0,1,\ldots,j\}.$
For any  $\vec{i}=(i(1),\ldots,i(k))\in P_k$ denote 
\[
T_{\vec{i}}=T_{i(1)}\circ\ldots \circ T_{i(k)}.
\]

For $l> k,$ let  $P_{k+1}^l=\prod_{j=k+1}^{l} \{1,2,\ldots,j\},$ 
and let $P_{k+1}^k$  contain only the constant sequence $(0,\ldots,0)$. 
Note that if $p\in\fin_l$ and $\vec{i}\in P^l_{k+1}$, then 
$T_{\vec{i}}(p)\in\fin_k$.

Let $l\geq k$ and let $B=(b_s)_{s=1}^m\in\fin_l^{*(m)}(n)$, 
we denote by $\left<\bigcup_{\vec{i}\in P_{k+1}^{l}} T_{\vec{i}}(B)\right>_{P_{k}}$ the partial subsemigroup of 
$\fin_{k}$ consisting of elements of the form
\[
\sum_{s=1}^m    T_{\vec{t}_s}    \circ T_{\vec{i}_s}(b_{s}),
\]
where $\vec{i}_1,\ldots,\vec{i}_m\in P_{k+1}^{l}$,  $\vec{t}_1,\ldots,\vec{t}_m\in P_k$,
and there is an $s$ such that all entries of $t_s$ are 0. 
By $\left<\bigcup_{\vec{i}\in P_{k+1}^{l}} T_{\vec{i}}(B)\right>_{P_{k}}^{*(d)}$, we denote  the set of all 
$(p_1,\ldots, p_d) \in\fin_k^{*(d)} $ such that $p_i\in \left<\bigcup_{\vec{i}\in P_{k+1}^{l}} T_{\vec{i}}(B)\right>_{P_{k}}$.

The following theorem is the combinatorial core of Theorem \ref{rams}, its proof was inspired by the proof a Ramsey Theorem in \cite{BLLM}.

\begin{thm}\label{lelek}
Let $k\geq 1$. Then for every $m\geq d,$ for every $ l\geq k,$ and for every $r\geq 2$ there exists a natural number $n$ such that for every colouring 
$c:\fin^{*(d)}_k(n)\to \{1,2,\ldots,r\},$ there is  $B\in\fin_l^{*(m)}(n)$  such that the partial semigroup 
$\left< \bigcup_{\vec{i}\in P_{k+1}^l}T_{\vec{i}} (B)\right>_{P_k}^{*(d)}$ is $c$-monochromatic. Denote the smallest such $n$ by 
$L(d,m,k,l,r).$

\end{thm}

In the proof of Theorem \ref{lelek}, we will use the Graham-Rothschild theorem about colouring of partitions \cite{GR}.
For natural numbers $d,n$ let $\mc P^d(n)$ denote the set of all partitions of the set $\{1,2,\ldots,n\}$ into exactly
$d$ non-empty sets. We say that a partition $\mathcal{P}$ is a {\em coarsening} of a partition $\mathcal{Q}$ if for every $Q\in\mathcal{Q}$ there is $P\in\mathcal{P}$ such that $Q\subset P$. 
Every partition  $\mathcal{P}\in \mc P^d(n)$ carries a {\em canonical enumeration}, where for $P, Q\in\mathcal{P}$ we set 
 $P<Q$ iff $\min(P)<\min(Q)$.
\begin{thm}[Graham-Rothschild, \cite{GR}]
Let $k<l $ and $r\geq 2$ be given natural numbers. Then there is a natural number~$n$ such that for any colouring of $\mc P^k(n)$ into $r$ colours
there is a partition $\mathcal{P}\in \mc P^l(n)$ such that the set 
$\{\mathcal{Q}\in \mc P^k(n): \mathcal{P} {\rm\ is\ a \ coarsening \ of\ \mathcal{Q}}  \}$ is monochromatic.
Let $\mathbf{GR}(k,l,r)$  denote the smallest such natural number $n$.
\end{thm}

\begin{proof}[Proof of Theorem \ref{lelek}]
We set $n=\mathbf{GR}(dk+1,ml+1,r)$ and
let $c:\fin^{*(d)}_k(n)\to \{1,2,\ldots,r\}$ be an arbitrary colouring.
Define the map $\Phi:\mc P^{dk+1}(n)\to \fin^{*(d)}_k(n)$ 
 that to a canonically enumerated partition $\mc P=(P_i)_{i=0}^{dk}$ assigns
\[
\Phi(\mc P)_j=\sum_{s=1}^k s\cdot \mathbbm{1}_{P_{(j-1)k+s}}
\]
for $j=1,\ldots,d.$ Then $c\circ\Phi$ is a colouring of $\mc P^{dk+1}(n)$.

Let a canonically enumerated partition $\mc Q= (Q_i)_{i=0}^{ml}$ be $c\circ\Phi$-monochromatic. We define $B=(b_j)_{j=1}^m$ by
\[
b_j=\sum_{s=1}^l s\cdot   \mathbbm{1}_{P_{(j-1)l+s}}.
\]

The following claim verifies that $\left< \bigcup_{\vec{i}\in P_{k+1}^l}T_{\vec{i}} (B)\right>_{P_k}^{*(d)}$ is contained
 in the image of coarsenings of $\mc Q$ of size $dk+1$ under $\Phi$, which implies that 
$\left< \bigcup_{\vec{i}\in P_{k+1}^l}T_{\vec{i}} (B)\right>_{P_k}^{*(d)}$ is $c$-monochromatic,  which is what we wanted to show.

\smallskip
\noindent {\bf{Claim.}}  
Let $A=(A_1,\ldots,A_d)\in \left< \bigcup_{\vec{i}\in P_{k+1}^l}T_{\vec{i}} (B)\right>_{P_k}^{*(d)}.$ Then 
$A=\Phi(\mathcal{P})$ for $\mc P=(P_j)_{j=0}^{dk}$ given by $P_{(s-1)k+i}:=\supp_i(A_s)$ for $s=1,\ldots,d$ and $i=1,\ldots,k, $ and $P_0=\{1,\ldots, n\}\setminus \bigcup_{j=1}^{dk} P_j.$

\begin{proof}
Clearly, $\mc P$ is a coarsening of $\mc Q$. 
Therefore all we have to show is that $\mc P$ is a partition into exactly $dk+1$ nonempty sets (see (1) below)
and that the enumeration of sets in $\mc P$ is the canonical enumeration (see (2), and (3), and (4)  below).
\begin{enumerate}
\item For every $i=1,\ldots, k$ and $s=1,\ldots, d$, we have $ \supp_i(A_s)\neq\emptyset$.
\item We have $Q_0\subset P_0$.
\item  For every $i,i'=1,\ldots, k$, $i<i'$, and $s=1,\ldots, d$
\[\min \supp_i(A_s)<\min \supp_{i'}(A_{s}).\]
\item  For every  $s=1,\ldots, d-1$,
\[\min \supp_k(A_s)<\min \supp(A_{s+1}).\]
\end{enumerate}

Property (2) is clear.
Properties (1) and (3) follow from the definition of $B$ and that we can write
\[
A_s=\sum_{j\in J_s} T_{\vec{i}_j^s} b_j
\]
for some $\vec{i}_j^s\in P_k P_{k+1}^l$ (where $P_k P_{k+1}^l$ is the set of concatenations of sequences in $P_k$ and $P_{k+1}^l$) and $J_s\subset\{1,\ldots, m\}$. 
Property (4) follows from $A\in\fin_k^{*(d)}(n).$

\end{proof}

\end{proof}

For a natural number $N$ we let  $N^{[j]}$ denote the collection of all $j$-element subsets of $\{1,\ldots, N\}$ and let
$N^{[\leq j]}$ denote the collection of all at most $j$-element subsets of $\{1,\ldots,N\}.$ 
Note that $N^{[\leq j]}=\bigcup_{i=0}^j N^{[j]}$.
We will often write $N$ instead of $N^{[1]}$. 

Let $m, k_1,\ldots,k_m, l_1,\ldots,l_m ,r\geq 2$, and $N$
be natural numbers
and let 
\[
c: \prod_{i=1}^m N^{[\leq k_i]}\to \{1,2,\ldots,r\}
\]
be a colouring.
Given $B_i\subset N$ for $i=1,2\ldots,m,$
we say that $c$ is {\em size-determined} on $(B_i)_{i=1}^m$ if
whenever $A_i, A_i'\subset B_i$ are such that $0\leq |A_i|=|A_i'|\leq k_i$ for $i=1,2,\ldots,m$ then
\[
c(A_1,\ldots,A_m)=c(A_1',\ldots, A_m').
\]

\begin{thm}[Theorem 4.10, \cite{BK2}]\label{Tdifheights}
Let $m, k_1,\ldots,k_m,  l_1,\ldots,l_m$ and $ r\geq 2$ be natural numbers such that $k_i\leq l_i$ for every $i=1,2,\ldots,m$. 
Then there exists $N$ such that for every colouring
\[
c: \prod_{i=1}^m N^{[\leq k_i]}\to \{1,2,\ldots,r\}
\]
there exist $B_1,\ldots, B_m\subset N$   with $|B_i|=l_i$ such that $c$ 
is size-determined on $(B_i)_{i=1}^m$.
Denote by $S(m,k_1,\ldots,k_m,l_1,\ldots,l_m,r)$ the minimal such $N.$
\end{thm}

For $f\in \prod_{i=1}^m N^{[\leq k_i]}$, we define $\supp(f)=\{i: f(i)\neq\emptyset\}$.
For a natural number~$d,$ let
$\left(\prod_{i=1}^m N^{[\leq k_i]}\right)^{*(d)}$ be the set of all sequences $(f_s)_{s=1}^d$
with $f_s\in \prod_{i=1}^m N^{[\leq k_i]}$ and 
$\supp(f_s)\cap \supp(f_{s+1})=\emptyset$, for each $s$. 
Note that supports of some of the $f_s$ may be empty.

Then, more generally, if
\[
\chi: \left(\prod_{i=1}^m N^{[\leq k_i]}\right)^{*(d)}\to \{1,2,\ldots,r\}
\]
is a colouring and $B_i\subset N$ for $i=1,2\ldots,m,$
we say that $\chi$ is {\em size-determined} on $(B_i)_{i=1}^m$ if whenever $(f_s)_{s=1}^d$ and $(g_s)_{s=1}^d$
are such that for each $i$ and $s$, $\supp(f_s)=\supp(g_s)$, $f_s(i),g_s(i)\subset B_i$, and $|f_s(i)|=|g_s(i)|$, then
\[\chi \left( (f_s)_{s=1}^d\right)=\chi \left( (g_s)_{s=1}^d\right).\]

Corollary \ref{dhe}  is a multidimensional version of Theorem \ref{Tdifheights}. 
\begin{cor}\label{dhe}
Let $d\leq m,$ and $k_1\leq l_1,\ldots,k_m\leq l_m,$ and  $ r\geq 2$ be natural numbers. 
 Then there exists $N$ such that for every colouring
\[
\chi: \left(\prod_{i=1}^m N^{[\leq k_i]}\right)^{*(d)}\to \{1,2,\ldots,r\}
\]
there exist $B_1,\ldots, B_m\subset N$   with $|B_i|=l_i$ such that $\chi$ 
is size-determined on $(B_i)_{i=1}^m$.
Denote by $S(d,m,k_1,\ldots,k_m,l_1,\ldots,l_m,r)$ the minimal such $N.$
\end{cor}
\begin{proof}  
Denote by $\Gamma$ the set of all ordered partitions $\gamma:\{1,\ldots, m\}\to\{1,\ldots, d\}$ into $d$ non-empty pieces.
Let  $c: \prod_{i=1}^m N^{[\leq k_i]}\to r^{\Gamma}$ be the colouring 
given by
\[ c(A_1,\ldots, A_m)(\gamma)=\chi\left((f^\gamma_s)^d_{s=1}\right), \text{ where }\]
 \[f_s^\gamma(n) = \begin{cases} 
A_n & \text{ if } n\in\gamma(s), \\
\emptyset & \text{otherwise}.  \end{cases} \]

Applying Theorem \ref{Tdifheights}, we get 
 $B_1,\ldots, B_m\subset N$   with $|B_i|=l_i$ such that $c$ is size-determined on $(B_i)_{i=1}^m$.
It follows that  $\chi$ is size-determined on $(B_i)_{i=1}^m$.
\end{proof}

The proof below is similar to the proof of Theorem 4.12 in Barto\v sov\'a-Kwiatkowska~\cite{BK2}.

\begin{proof}[Proof of Theorem \ref{rams}]
Let $S\in \f_{cc}$ be of height $k$ and width $d$, and let 
 $T\in \f_{cc}$ be of height $l\geq k$  and width $m\geq d$ (so that ${T\choose S}\neq \emptyset$). Let $r\geq 2$ be the number of colours. Let $n$ be as in 
Theorem \ref{lelek} 
for $d,m,k,l,r$, that is $n=L(d,m,k,l,r)$, and let $N$ be as in Corollary \ref{dhe} for 
$d, n,k,\ldots,k,l,\ldots,l,r$, that is $N=S(d,n,k,\ldots,k,l,\ldots,l,r)$. 
Let $U\in\f_{cc}$ consist of $n$ branches of height $N.$ We will show that this $U$ works for $S,T$ and $r$ colours.

Let $a_1,\ldots, a_d$ and $c_1,\ldots,c_n$ be the  increasing (according to $<^S_{cc}$ and $<^U_{cc}$) enumerations of branches in $S$ and $U$, respectively.
Let $(a_j^i)_{i=0}^k$ be the increasing enumeration of the branch $a_j$, $j=1,\ldots,d$, 
  and let  
$(c_j^i)_{i=0}^N$  be the increasing enumeration of the branch $c_j$ for $j=1,\ldots,n.$

To each $f\in {U\choose S}$, we associate  $f^*=(p^f_i )_{i=1}^d\in \fin_k^{*(d)}(n)$ such that 
\[
		\supp(p^f_i)=\{j: a^1_i\in f(c_j)\}
	\]
		and for $j\in \supp(p_i^f)$
\[
		p_i^f(j)=z 
		\ \iff \ f(c_j^N)=a_i^z.
\]

We moreover associate to $f$ a sequence $(F_i^f)_{i=1}^d\in (\prod_{j=1}^n  (c_j\setminus \{c_j^0\})^{[\leq k]})^{*(d)}$ such that 
for each $i$ there is $j$ with $F_i^f(j)\in (c_j\setminus \{c_j^0\})^{[k]}$ as follows. For $j\in\supp(p_i^f),$ we let
 \[ F_i^f(j)=\{\min\{c_j^y\in c_j : f(c_j^y)=a_i^x\} : 0<x\leq p^f_i (j)\},\]
where the $\min$ above is taken with respect to the partial order $\preceq_U$ on the fan $U$.

Let us remark that $f\mapsto (F_i^f)_{i=1}^d$ is an injection from ${U\choose S}$ to 
$(\prod_{j=1}^n(c_j\setminus \{c_j^0\})^{[\leq k]})^{*(d)}$ and
 $f\mapsto f^*$ is a surjection from ${U\choose S}$ to $\fin_k^{*(d)}(n)$. Note that if $f_1^*=f_2^*$ then 
$|F_i^{f_1}(j)|=|F_i^{f_2}(j)|$ for all $i,j.$

Analogously, to any $g\in{U\choose T}$, we associate $g^*\in\fin_l^{*(m)}(n)$ and
$(F_i^g)_{i=1}^m\in (\prod_{j=1}^n (c_j\setminus \{c_j^0\})^{[\leq l]})^{*(m)}$.

Let $e: {U\choose T}\to \{1,\ldots, r\}$ be a colouring.
Let $e_0$ be a colouring of $(\prod_{j=0}^n (c_j\setminus \{c_j^0\})^{[\leq k]})^{*(d)}$ 
induced by the colouring $e$ via
 the injection $f\to (F_i^f)_{i=1}^d$. We colour elements in $(\prod_{j=0}^n (c_j\setminus \{c_j^0\})^{[\leq k]})^{*(d)}$ 
not of the form $(F_i^f)_{i=1}^d$ in an arbitrary way by one of the colours in $\{1,\ldots, r\}$.
Applying Corollary \ref{dhe}, we can find $C_j\subset c_j\setminus \{c_j^0\}$ of size $l$ for  $j=1,\ldots,n$ such that $e_0$ 
is size-determined on $(C_j)_{j=1}^n$.  
It follows that the colouring $e^*:\fin_k^{*(d)}(n)\to \{1,2,\ldots,r\}$ given by $e^*(f^*)=e(f)$ for  $f\in {U\choose S}$ with
 $(F_i^f)_{i=1}^d\in (\prod_{j=1}^n C_j^{[\leq k]})^{*(d)}$ 
 is well-defined. 
 
 Now we can  apply Theorem \ref{lelek} to get  $D=(d_j)_{j=1}^m\in\fin_l^{*(m)}(n)$  such that 
 $\left<\bigcup_{\vec{i}\in P_{k+1}^l} T_{\vec{i}}(D)\right>^{*(d)}_{P_k}$ is $e^*$-monochromatic.
 Let $g\in{U\choose T}$ be any epimorphism such that  
 $(F_i^g)_{i=1}^m\in (\prod_{j=1}^n C_j^{[\leq l]})^{*(m)}$ and $g^*=D$.
Then for every $h\in{T\choose S}$, we have  $(h\circ g)^*\in \left<\bigcup_{\vec{i}\in P_{k+1}^l} T_{\vec{i}}(D)\right>^{*(d)}_{P_k}$.
Since  $\left<\bigcup_{\vec{i}\in P_{k+1}^l} T_{\vec{i}}(D)\right>^{*(d)}_{P_k}$ is $e^*$-monochromatic,
we  conclude that ${T\choose S}\circ g$ is $e$-monochromatic.
\end{proof}

\subsection{The  universal minimal flow of $\boldsymbol{\aut(\lel)}$}\label{umfa}
In this section, we present two descriptions of the universal minimal flow of $\aut(\lel)$ and we exhibit an explicit isomorphism between them.

Let $X^*$ be the set of downwards closed maximal chains on $\lel$ equipped with the topology inherited from $\Exp(\Exp(\lel))$.
 This is a compact   space, which we proved in Proposition \ref{ccompact}.
The group $\aut(\lel)$ acts on $X^*$ by left translations
\[g\cdot \c_1=\c_2 \  \ \iff \ \ \c_2=\{g(C): C\in \c_1\}.\]

\begin{prop}\label{cont}
The action $
\aut(\lel)\acts X^*$ is continuous.
\end{prop}
\begin{proof}
We have the following general fact: Whenever a Polish group $H$ acts continuously on a compact  space $K$, then the 
corresponding action of $H$ on ${\rm Exp( K)}$ by left translations is also continuous (see page 20 in \cite{BKe}, Example (ii) ).

It follows, using the fact above  twice, that the action of $\aut(\lel)$ on ${\rm Exp(Exp}(\lel))$ by left translations is continuous. 
Therefore the restriction of this action to the closed invariant set $X^*$ is also continuous, which is what we wanted to show.
\end{proof}

In this section, we prove the following:
\begin{thm}\label{umfauto}
The universal minimal flow of $\aut(\lel)$ -- the automorphism group of the projective Fra\"{i}ss\'{e} limit of finite  fans -- is equal to
\[
\aut(\lel)\acts X^*
\]
and it is isomorphic to
\[ 
\aut(\lel)\acts\widehat{\aut(\lel)/\aut(\lel_c)}.
\]
\end{thm}

 Let
\begin{equation*}
\begin{split}
X_{\lel_c}=&\{\c\in X^*:
{\rm\ for\ every\ } A\in\f 
{\rm\ and\ an\ epimorphism\ }
\phi: \lel\to A {\rm\ there \ exists\ } \\ 
& A_c\in\f_c, {\rm such \ that\ } \phi: (\lel,\c)\to A_c
{\rm \ is\ an\ epimorphism }\}.
\end{split}
\end{equation*}

We make  $X_{\lel_c}$ a topological space by declaring sets
\[ V_{\phi, A_c}=\{\c\in X_{\lel_c} : \phi: (\lel,\c)\to A_c  \text{  is an epimorphism} \}, \]
where $\phi: \lel\to A$ is an epimorphism,  $A_c\in\f_c$, and $A_c\rest \ll= A$,  to be open.

\begin{prop}\label{ident2}
We have $X_{\lel_c}=X^*$.
\end{prop}

\begin{proof}
We have to show two things: $X^*\subset X_{\lel_c}$ and the topologies on  $X^*$ and $X_{\lel_c}$ agree.

 Lemma \ref{image3} implies that if $\c\in X^*$ and  $\phi: \lel\to A$ is an epimorphism, then
there is a (necessarily unique) $A_c\in\f_c$ with
$A_c\restriction \ll= A$ such that $\phi: (\lel,\c)\to A_c$ is an epimorphism, from which
it follows that  $X^*\subset X_{\lel_c}$.




Since   $X^*$ and $X_{\lel_c}$  are both compact, it suffices to show that the identity map from $X^*$ to $X_{\lel_c}$ is continuous.
For this we show that sets of the form $V_{\phi, A_c}$ are open in $X^*$. 
For any partition $P$ of $\lel$ into clopen sets and any 
 $P_1,\ldots, P_n\subset P$, by the definition of the Vietoris topology, 
 each of the sets
\begin{align*} 
 \overline{P}_i= & \{ D\in   {\rm Exp}( \lel): D\subset\bigcup  P_i \text{ and for every } p\in P_i,\  D\cap p\neq\emptyset \} \\
 =&\{ D\in   {\rm Exp}( \lel): P_i =\{p\in P: D\cap p\neq\emptyset\}   \}  
 \end{align*}
is open in  ${\rm Exp}( \lel)$ and therefore
\begin{align*}
V_{P_1,\ldots,P_n}=&  
\{\c\in { \rm Exp( Exp(}\lel)) :    \c\subset \overline{P}_1\cup\ldots \cup\overline{P}_n
 \text{ and for every } i,\  \c\cap\overline{P}_i\neq\emptyset \} \\
=&\{\c\in { \rm Exp( Exp(}\lel)) :     \\
& \{P_1,\ldots,P_n\}=\{Q\subset P : \text{ for some } D\in\c, \ Q=\{p\in P: D\cap p\neq\emptyset\}        \}  \  \}
\end{align*}
is open in ${ \rm Exp( Exp(}\lel)) $.
Now if $P=\{\phi^{-1}(a): a\in A\}$, where $\phi: \lel\to A$ is an epimorphism,
and if $P_i =\{\phi^{-1}(a_j): a_j\leq^{A_c} a_i\}$, where $A=\{a_1,\ldots, a_n\}$ is linearly ordered by $\leq^{A_c}$,
the  order on $A$ induced by the chain on $A_c$, then 
$ V_{\phi, A_c}=V_{P_1,\ldots,P_n}\cap X^*$, which implies that $ V_{\phi, A_c}$ is open.


\end{proof}

\begin{lemma}\label{minim}
The family $\f_c$ has  the expansion property with respect to $\f$,
that is, 
 for any $A_c\in\f_c$ there is $D\in\f$ such that
for any expansion $D_c\in\f_c$ of $D$, there is an epimorphism $\phi: D_c\to A_c$.
\end{lemma}

\begin{proof}
 Using Proposition \ref{coin}, find $B_c\in\f_{cc}$ such that there is an epimorphism from $B_c$ onto $A_c$.
Let $k$ and $l$ be the height and width of $B_c$, respectively. We show that $D\in\f$  of height $m=kl$ and width $l$ is as required, that is,
for any $D_c\in\f_c$ with $D_c\rest \ll= D$, there is an epimorphism from $D_c$ onto $B_c$.

Let $b_1 <^{B_c}_{cc} b_2 <^{B_c}_{cc}\ldots <^{B_c}_{cc} b_l$ be the increasing enumeration of branches of $B_c$ 
 and
 let  $d_1, d_2, \ldots, d_l$ be a list of all branches of $D$ . 
 Let for each $i$, $(d_i^j)_{j=0}^m$ be the $\preceq_D$-increasing enumeration of the branch $d_i$
 and let  for each $i$, $(b_i^j)_{j=0}^k$ be the $\preceq_B$-increasing enumeration of the branch $b_i$.
 Let $v_B$ and $v_D$ denote the roots of $B$ and $D$, respectively.

Take any $D_c\in\f_c$ with $D_c\rest \ll= D$. We recursively define a sequence $(x^i)_{i=1}^l$ as follows.
Let $x^i$ be the least (with
respect to the linear order $\leq^{D_c}$ induced by $D_c$) 
element
from the set $\{d_j^{ik}: j=1,2,\ldots, l\}$
which is not on the same branch as $x^1,\ldots, x^{i-1}$ are. Let $\bar{d}_i$ denote the branch of $D$ on which $x^i$ is, so $x^i=\bar{d}_i^{ik}$.
Let $\psi: D_c\to B_c$ be the $R$-preserving map satisfying   for each $i$,
\begin{align*}
&\psi(\bar{d}_i^{(i-1)k+t})=b_i^t, \ t=1,2,\ldots, k,\\
&\psi([\bar{d}_i^{ik}, \bar{d}_i^m]_{\preceq_D})=b_i^k, \text { and} \\
&\psi([v_D, \bar{d}_i^{(i-1)k}]_{\preceq_D})=v_B.
\end{align*}
Then  $\psi$ preserves chains and therefore it is a required epimorphism. 
\end{proof}


\begin{proof}[Proof of Theorem \ref{umfauto}]
We will apply Theorem \ref{kpt2}, via Theorem \ref{stone2},  to families $\f$ and $\f_c$
to show that the flow $\aut(\lel)\acts X_{\lel_c}$ is the universal minimal flow of $\aut(\lel)$.
Clearly, $\f_c$ is a  precompact expansion of $\f$ and  from  Lemma \ref{image3} it follows that  the property~$(*)$ holds for $\f$ and $\f_c$.
Further, the expansion $\f_c$ of $\f$ is reasonable (Lemma \ref{reason}),  $\f_c$ has  the expansion property with respect to $\f$ (Lemma \ref{minim}), and 
$\f_{c}$ is a Ramsey class (Theorem \ref{rams} and Corollary \ref{xyza}).

Finally, Corollary \ref{ident} implies that $\aut(\lel)\acts \widehat{\aut(\lel)/\aut(\lel_c)}$
and Proposition \ref{ident2}  implies that $\aut(\lel)\acts X^*$ describe the universal minimal flow of $\aut(\lel)$.

\end{proof}

\section{The universal minimal flow of $H(L)$}\label{umfh}

We will compute the universal minimal flow of $H(L)$ -- the homeomorphism group of the Lelek fan $L$ -- in
  two ways, as we did for $\aut(\lel)$; one description will correspond to the 
quotient description, and the other to the chain description.

\subsection{The 1st description}
Let $\pi:\lel\to L$ denote the continuous surjection and let $\pi^*: \aut(\lel)\to H(L)$ denote the  continuous homomorphism with a dense image,
both  defined in Section~\ref{sec:lelek}. 
The chain $\c^L=\pi(\c^\lel)$ is downwards closed and it is maximal by Lemma \ref{image3}. Let $L_c=(L, \c^L)$ and 
set \[H(L_c)=\{h\in H(L): h(\c^L)=\c^L\}.\] 
Note that $H(L_c)$ is closed in $H(L)$. 
Indeed, since for any $D\in {\rm Exp}(L)$  the map that takes $h\in H(L)$
and assigns to it $h(D)\in {\rm Exp}(L)$ is continuous, 
and since $\c^L$ is maximal  and hence closed in $\Exp(L)$, we get that $H(L_c)$ is closed .

In this section, we will prove the following theorem.
\begin{thm}\label{umfhom1}
 The universal minimal flow of $H(L)$ is equal to
$H(L)\acts \widehat{H(L)/H(L_c)}$.
\end{thm}

Theorem \ref{image} below will immediately imply that the universal minimal flow of $H(L)$ is equal to
$\widehat{H(L)/H_1}$, where $H_1=\overline{\pi^*(\aut(\lel_c))}$.
We will identify $H_1$ with $H(L_c)$ in Proposition \ref{hl}.

First, we need Theorem \ref{unii}, its proof is in \cite{NVT} (the proof of $ii)\to i)$ of Theorem 5).
We will say that a flow $G\acts X$ is {\em universal} in a family of $G$-flows $\mathcal{FL}$ if $G\acts X\in \mathcal{FL}$ and
for every $G$-flow $G\acts Y\in \mathcal{FL}$
there is  a continuous $G$-map from $X$ onto $Y$.
\begin{thm}\label{unii}
Let $G$ be a Polish group and let $H$ be its closed subgroup. Suppose that $H$ is extremely amenable and that $G/H$ is precompact. Then 
the $G$-flow $G\acts\widehat{G/H}$ is universal in the family of $G$-flows in which there is an $H$-fixed point with a dense orbit.
\end{thm}

\begin{thm}\label{image}
Let $G$ be a Polish group and let $H$ be a closed subgroup of $G$. Suppose that $H$ is extremely amenable,
$G/H$ is precompact, and $M(G)=\widehat{G/H}$. 
Let $G_1$ be a Polish group and let $\phi: G\to G_1$ be a continuous homomorphism with a dense image. 
Then $M(G_1)=\widehat{G_1/H_1}$, where $H_1=\overline{\phi(G)}$.
\end{thm}
\begin{proof}
First note that $H_1$ is extremely amenable (see Lemma 6.18 in \cite{KPT}) and that $\phi$ is uniformly continuous.
  We show that $G_1/H_1$ is precompact. 
Consider $\tilde{\phi}: G/H\to G_1/H_1$ given by $\tilde{\phi}(gH)=\phi(g)H_1$. This map is well defined,  uniformly 
continuous (with respect to quotients of the right  uniformities),
 and has a dense image. As the image of a precompact space by a uniformly continuous map is precompact
(a straightforward calculation, this property is also stated in Engelking \cite{E}, page 445, the 2nd paragraph),
 $\tilde{\phi}(G/H)$ is precompact in the uniformity inherited from $G_1/H_1$. Moreover, as $\tilde{\phi}(G/H)$ is dense
 in $G_1/H_1$, completions of both spaces are equal (see \cite{E} 8.3.12). This gives that $G_1/H_1$ is precompact.

Since $\tilde{\phi}$ is  uniformly continuous, $\tilde{\phi}$ extends to a $G$-map 
$\tilde{\Phi}: \widehat{G/H}\to \widehat{G_1/H_1}$ (see 8.3.10 in \cite{E}). 
As  $\tilde{\phi}(G/H)$ is dense in $G_1/H_1$ and the image $\tilde{\Phi}(\widehat{G/H})$ is closed
in  $\widehat{G_1/H_1}$, $\tilde{\Phi}$ is onto.
The continuous action $G_1\acts \widehat{G_1/H_1}$ and $\phi: G\to G_1$ induce a continuous action
$G\acts \widehat{G_1/H_1}$. As $G\acts \widehat{G/H}$ is minimal,
so is $G\acts \widehat{G_1/H_1}$, therefore the flow $G_1\acts \widehat{G_1/H_1}$ is minimal as well.
By Theorem \ref{unii},  the flow $G_1\acts \widehat{G_1/H_1}$ is universal in the family of minimal $G_1$-flows.

\end{proof}


To finish the proof of Theorem \ref{umfhom1}, we  show the proposition below.
\begin{prop}\label{hl}
We have $H_1=H(L_c)$.
\end{prop}
To show Proposition \ref{hl}, we need Lemma \ref{hl2}, which generalizes the following lemma.
\begin{lemma}[Barto\v sov\'a-Kwiatkowska \cite{BK}, Lemma 2.14]\label{coverold}
Let $d<1$ be any metric on~$L$.
Let $\epsilon>0$ and let   $v$ be the root of $L$. Then there is $A\in\f$ and  
an open cover $(U_a)_{a\in A}$ of $L$ 
such that  
\begin{itemize}
\item[(C1)] for each $a\in A$, $\text{diam}(U_a)<\epsilon$, 
\item[(C2)] for each $a,a'\in A$, if $U_a\cap U_{a'}\neq\emptyset$ then $R^{A}(a,a') $ or $R^{A}(a', a) $,
\item[(C3)] for each $x,y\in L$ with $x\in [v,y]$, if $x\in U_a$ and $y\in U_b$,  $a\neq b$, and  $\{x,y\}\not\subset U_a\cap U_b$, then $a\preceq_A b$, 
where $\preceq_A$ is the partial order on $A$, 
\item[(C4)] for every $a\in A$ there is $x\in  L$ such that $x\in U_a\setminus(\bigcup \{U_{a'}: a'\in A, a'\neq a\}).$
\end{itemize}
\end{lemma}

\begin{lemma}\label{hl2}
Let $d<1$ be any metric on $L$.
Let $\epsilon>0$ and let   $v$ be the root of $L$. Then there is $A_c=(A,\c^A)\in\f_c$ and  
an open cover $(U_a)_{a\in A_c}$ of $L_c$ 
such that (C1)-(C4) from Lemma \ref{coverold} hold and additionally we have the following.

{\rm (C5)}  For any $c\in \c^L$, the set $\{a\in A: c\cap U_a\neq\emptyset\}$  is in $\c^A$.
\end{lemma}
\begin{proof}
Take the cover $(U_a)_{a\in A}$ of $L$ as in Lemma \ref{coverold}. Then  the set
 $\{\{a\in A: c\cap U_a\neq\emptyset\}: c\in\c^L\}$ is a downwards closed chain in $A$.
To get $A_c$ extend this chain to a  downwards closed maximal one.
\end{proof}

\begin{proof}[Proof of Proposition \ref{hl}]
Since
$H(L_c)$ is closed, it follows that $H_1\subset H(L_c)$. 

To show the converse, take $h\in H(L_c)$ and $\epsilon>0$. 
Let $d<1$ be any metric on $L$ and let $d_{\sup}$ be the corresponding supremum metric on $H(L)$.
We will find $\gamma\in\aut(\lel_c)$ such that $d_{\sup}(h,\gamma^*)<\epsilon$, which will finish the proof as $\gamma^*\in H_1$
and since $H_1$ is closed.
Let $A_c\in\f_c$ and  $(U_a)_{a\in A}$, an open cover of $L$, be as in Lemma \ref{hl2} taken for $d$ and $\epsilon$.
Since $h$ is  uniformly continuous, we can assume additionally that for each $a\in A$, $\text{diam}(h(U_a))<\epsilon$.
Let $(V^1_a)_{a\in A}$ and $(V^2_a)_{a\in A}$ be the open covers of $\lel$
given by $V^1_a=\pi^{-1}(U_a)$ and  $V^2_a=\pi^{-1}(h(U_a))$. 

 For $B^i\in\f$ and an epimorphism $\phi^i: \lel\to B^i$ such that $\{(\phi^i)^{-1}(b): b\in B^i\}$ refines $(V^i_a)_{a\in A}$, $i=1,2$, denote by
$(W^i_a(\phi^i))_{a\in A}$ the clopen partition of $\lel$ such that for every~$a$, $W^i_a(\phi^i)$ is the union of all $(\phi^i)^{-1}(b)$ which lie in $V^i_a$
and do not lie in a $V^i_{a'}$ for some $a'$ with $R^A(a,a')$. This  partition defines an epimorphism $\psi_{\phi^i}$ from $\lel$ to $A$.
Indeed, by (C4) $\psi_{\phi^i}$ is onto.
 The properties (C2) and (C3) imply that if $x,y\in \lel$ satisfy $R^{\lel}(x,y)$ then 
$R^{A}(\psi_{\phi^i}(x),\psi_{\phi^i}(y))$,  which already gives that $\psi_{\phi^i}:\lel\to A$ is an epimorphism.
Observe that if we take arbitrary  $B^i\in\f$ and epimorphism $\phi^i:\lel\to B^i$, $i=1,2$, such that $\{(\phi^i)^{-1}(b): b\in B\}$ refines $(V^i_a)_{a\in A}$, then
for any $c\in\c^\lel $ and $a\in A$, we  have 
$a\in\psi_{\phi^i}(c)$ iff $c\cap W_a(\phi^i)\neq\emptyset$. 

Define $\d=\{\{a: c\cap V^i_a\neq\emptyset:\}: c\in\c^\lel\}$ and
observe that $\mathcal{D}=\{\{a: c\cap U_a\neq\emptyset:\}: c\in\c^L\}$, and hence  $\mathcal{D}\subseteq \c^A$, by the definition of $\c^A$.
This
follows for $i=1$ from that  for any $c\in\c^\lel $ and $a\in A$, 
$c\cap V^1_a\neq\emptyset$ iff $c\cap \pi^{-1}(U_a)\neq\emptyset$ iff $\pi(c)\cap U_a\neq\emptyset$;
and  for $i=2$ it follows from that  for any $c\in\c^\lel $ and $a\in A$, $c\cap V^2_a\neq\emptyset$ iff $c\cap \pi^{-1}(h(U_a))\neq\emptyset$ iff 
$\pi(c)\cap h(U_a)\neq\emptyset$ iff $h^{-1}(\pi(c))\cap U_a\neq\emptyset$, and use that $h^{-1}$ preserves $\c^L$.


\smallskip
\noindent{\bf{Claim.}} There are $B^i\in\f$  and epimorphisms $\phi^i:\lel\to B^i$, $i=1,2$, such that $\{(\phi^i)^{-1}(b): b\in B^i\}$ refines $(V^i_a)_{a\in A}$,
  satisfying $\c^A=\psi_{\phi^i}(\c^\lel)$. 

\smallskip

\begin{proof}[Proof of the Claim]
We fix $i=1,2$. 
Let $k$ be the number of elements in $A$ and let
 $\d$ be as before. Note that $\d$ may not be maximal.
Inductively on $1\leq m\leq k$ we construct  $B_m\in\f$ together with an epimorphism $\phi_m:\lel\to B_m$
such that the partition $P_m=\{\phi_m^{-1}(b): b\in B_m\}$ refines $P_{m-1}$ and
 the the first $m$ links of the maximal chain $\psi_{\phi_m}(\c^\lel)$ are equal to the first $m$ links of $\c^A$. 
 Then  $\phi^i=\phi_k$ will be as required. 

We let $\phi_1:\lel\to B_1$  be an  arbitrary epimorphism such that $\{\phi_1^{-1}(b): b\in B_1\}$ refines $(V^i_a)_{a\in A}$.
Suppose that we already constructed $\phi_1:\lel\to B_1,\ldots, \phi_m:\lel\to B_m$, $m<k$, and   
that there is an $m$-element link $M$ in $\d$.
We construct $\phi_{m+1}:\lel\to B_{m+1},\ldots, \phi_{m+l}:\lel\to B_{m+l}$, where 
$l>0$ is the smallest such that there is an $(m+l)$-element link $N$ in $\d$.
If $l=1$ and the $(m+1)$-st link of the maximal chain $\psi_{\phi_m}(\c^\lel)$ is equal to the  $(m+1)$st link of $\c^A$
(which may not be the case even if $l=1$), we let $\phi_{m+1}=\phi_m$.

Otherwise, let us see that \[S=\{d\in\c^\lel: N=\{a\in A: d\cap V^i_a\neq\emptyset\} \text{ and } \psi_{\phi_m}(d)=M\}\] 
has at least $l$ elements.

If $|M|+1=|N|$ and $S=\emptyset$, then using that for every $d\in \c^\lel$, $ \psi_{\phi_m}(d)\subseteq \{a\in A: d\cap V^i_a\neq\emptyset\} $, we obtain that 
the $(m+1)$-st link of $\psi_{\phi_m}(\c^\lel)$ is equal to the  $(m+1)$-st link of $\c^A$ (which in this case is equal to $N$). 
But we explicitly ruled out this happening.

If $|M|+2\leq |N|$, we will show that $S$ is infinite. For this it will suffice to show that there is no smallest $d\in\c^{\lel}$ (with respect to the inclusion) such that 
$N=\{a\in A: d\cap V_a^i\neq\emptyset\}$ and that simultaneously there is the smallest $d\in\c^{\lel}$ such that $\psi_{\phi_m}(d)=M$.
The second claim is clear, since $B_m$ is a clopen partition. For the first claim, suppose towards a contradiction that there is the smallest $d_s\in\c^\lel$ 
 such that $N=\{a\in A: d_s\cap V^i_a\neq\emptyset\}$. Clearly  there is the largest
 $d_l\in\c^\lel$  such that $M=\{a\in A: d_l\cap V^i_a\neq\emptyset\} $ and $d_l\subseteq d_s$.
However, this contradicts the maximality of $\c^\lel$. If  $a_1\neq a_2\in N\setminus M$, then 
$d_s\setminus V^i_{a_2}$ is downwards closed, properly contained in $d_s$, and it properly contains $d_l$.

Therefore we can find $c_1\subset \ldots\subset c_l\in\c^\lel$ with $c_j\in S$.
 Let $M=M_0\subset\ldots\subset M_l=N$ be links in $\c^A$ such that  $|M_{j+1}\setminus M_j|=1$.
 We can now successively define $ \phi_{m+1},\ldots, \phi_{m+l}$  in a way that 
$\psi_{\phi_{m+j}}(c_j)=M_j$ and 
the first $m+j$ links of the maximal chain $\psi_{\phi_{m+j}}(\c^\lel)$ are equal to the first $m+j$ links of $\c^A$,
$j=1,\ldots, l$. To get $\phi_{m+j}$, we take $a\in M_j\setminus M_{j-1}$ and we pick $p\in P_{m+j-1}$ such that $p\cap c_j\neq \emptyset$,
$p\cap V_a^i\neq\emptyset$ and $p\cap V_{a^-}^i\neq\emptyset$, where $a^-\in A$ is such that $R^A(a^-,a)$. We then partition $p$
into clopen sets $r$ and $s$ such that $s\subset V_a^i$ and 
$\phi_{m+j}$ with $P_{m+j}=(P_{m+j-1}\setminus \{p\})\cup \{r,s\}$ is an epimorphism. Then $\phi_{m+j}$ is as required.
 We are done with steps $m+1$,...,$m+l$. 
 \end{proof}
Denote $\psi_1=\psi_{\phi^1}$ and $\psi_2=\psi_{\phi^2}$. 
The (L3)
 provides us with $\gamma\in\aut(\lel_c)$ such that $\psi_1=\psi_2\circ \gamma .$
 We will show that $d_{\sup}(h,\gamma^*)<\epsilon$. 
Pick any $x\in \lel$, let $a=\psi_1(x)$, and note that 
\[ \gamma^*(\pi(x))\in\gamma^*(\pi\circ\psi_1^{-1}(a))=\pi\circ\gamma\circ\psi_1^{-1}(a)=\pi\circ\psi_2^{-1}(a)\subset h(U_a).\] 
Therefore $\gamma^*(\pi(x))\in h(U_a)$, $h(\pi(x))\in h(U_a)$, and 
 $\text{diam}(h(U_a))<\epsilon$, and we get the required conclusion.

 \end{proof}

\subsection{The 2nd description}
Let $C$ be the Cantor set viewed as the middle third Cantor set in [0,1].
Each point of $C$ can be expanded in a ternary sequence $0.a_1a_2 a_3\ldots$, where $a_i\in\{0,2\}$ for every $i$,
and each point of the interval $[0,1]$ can be expanded in a binary sequence $0.b_1b_2 b_3\ldots$ where $b_i\in\{0,1\}$ for every $i$.
Let $\pi_0: C\to [0,1]$ be given by $\pi_0(0.a_1a_2 a_3\ldots)=0.b_1b_2 b_3\ldots$, where $b_i=0$ if $a_i=0$ and $b_i=1$ if $a_i=2$.
We can view the Cantor fan $F$ as the union of segments joining the point $v=(\frac{1}{2},0)\in\mathbb{R}^2$ and a point $(c, 1)\in\mathbb{R}^2$, where $c\in C$.
We first describe a topological $\mathcal{L}$-structure $\mathbb{F}$  such that $\pi(\mathbb{F})=F$, where $\pi$ is be the continuous surjection such that 
$\pi(x)=\pi(y)$ if and only if $R^{\mathbb{F}}(x,y)$, and for every $(a,b)\in\mathbb{F}$ the second coordinate of $\pi(\mathbb{F})$ is equal to $\pi_0(b)$.
For this we let $\mathbb{F}= F\cap (C\times C)$ and 
$R^{\mathbb{F}}((a,b),(c,d)) $ if and only if $a=c$ and $b=d$, or if $v$, $(a,b)$, and $(c,d)$ are collinear and
 $(b,d)$ is an interval removed from $[0,1]$ in the construction of $C$.
 
 We view the Lelek fan $L$ as a subset of $F$ with its root equal to  $v$
 and we notice that $\lel$ is isomorphic to $\pi^{-1}(L)$.

  Let $m_0$ denote the metric on $L$, equal to the restriction of the Euclidean distance on $\mathbb{R}^2$.
 Let $d_0$ be the corresponding supremum metric on $H(L)$. It is a right-invariant metric and it induces a right-invariant metric $d$ on $H(L)/H(L_c)$ via 
 \[ d(gH(L_c), hH(L_c))=\inf\{d_0(gk, h): k\in H(L_c)\}.  \]

   Let $m$ be the metric on $\Exp(\Exp(L))$ obtained from  $m_0$. 
 To be more specific, to get $m$, we first take  the Hausdorff metric on $\Exp(L)$ with respect to $m_0$ and then we 
take the Hausdorff metric on $\Exp(\Exp(L))$ with respect to that metric on $\Exp(L)$.

Let $Y^*$ be the set of downwards closed maximal chains on $L$. Note that $\pi:\lel\to L$ induces a continuous surjection from $\Exp(\lel)$ to $\Exp(L)$,
which further  induces a continuous surjection from $\Exp(\Exp(\lel))$ to $\Exp(\Exp(L))$.
Let $\pi_c: X^*\to Y^*$ be the restriction of this last map to $X^*$.
Note that $\pi_c$ is onto. 
 Indeed, take $\mathcal{D}\in Y^*$ and observe that $\pi_c^{-1}(\mathcal{D})$ is a downwards closed chain in $\lel$.
 Using Zorn's Lemma, extend it to a downwards closed maximal chain $\mathcal{C}$, and note that $\pi_c(\mathcal{C})=\mathcal{D}$. 
Consider the $H(L)$-flow $H(L)\acts Y^*$ induced from the natural action of $H(L)$ on $L$ given by $(g,x)\to g(x)$. In Theorem \ref{umfhom2}
 we will show that this flow is isomorphic 
to the flow $H(L)\acts \widehat{H(L)/H(L_c)}$. The main ingredient of the proof will be Theorem \ref{equival}. 

\begin{thm}\label{umfhom2}
The universal minimal flow of $H(L)$ is isomorphic to  $H(L)\acts Y^*$. 
\end{thm}

  Let $G=H(L)$ and let $H=H(L_c)$. 
 \begin{thm}\label{equival}
 The bijection 
 \[   gH \ \to \   g\cdot \c^L \]
 is  a uniform  $G$-isomorphism from $G/H$ to $G\cdot \c^L$.
 \end{thm}
 
 Let us first see that Theorem \ref{equival} implies Theorem \ref{umfhom2}.
 \begin{proof}[Proof of Theorem \ref{umfhom2}]
The uniform $G$-isomorphism $gH \to  g\cdot \c^L$ from  $G/H$ to $G\cdot \c^L$  extends to 
 a uniform $G$-isomorphism $h$ between the completion $\widehat{G/H}$ of  $G/H$ and the completion of $G\cdot \c^L$,
 which is the closure  of $G\cdot \c^L$ in $Y^*$, which we have to show is equal to $Y^*$.
 
 Let  $f:  \widehat{\aut(\lel)/\aut(\lel_c)}\to X^*$ be the uniform $G$-isomorphism that extends the map $g \aut(\lel_c)\to g\cdot \c^\lel$ , let 
 $\rho$ be the extension  of the map $g \aut(\lel_c)\to gH$ to the respective completions $\widehat{\aut(\lel)/\aut(\lel_c)}$
 and $\widehat{G/H}$.
 Let $\pi_c: X^*\to Y^*$, as before, be the map induced from $\pi:\lel\to L$.  
As  $\pi_c f= h\rho$ and  $\pi_c$ is onto, we obtain that $h$ is onto.

\end{proof}


We define the {\em mesh} of a finite open cover $\u$ of $L$ to be the maximum of diameters of the sets in $\u$ and denote it by $\mesh(\u)$,
and we define the  {\em spread} of $\u$ to be the minimum of distances between non-intersecting sets in $\u$ and denote it by $\spr(\u)$

 Let $A\in\f$ and 
let $\mu_0^A$ 
be the {\em path metric} on $A$ in which $\mu_0^A(x,y)$ is equal to the length of the shortest path joining $x$ and $y$  in the undirected graph obtained from $A$
via symmetrization of the relation $R^A$. The {\em length} of a path  is understood to be the number of edges in the path. The metric $\mu_0^A$ on $A$ induces the Hausdorff
 metric $\mu_1^A$ on $\Exp(A)$, which in turn induces the Hausdorff metric $\mu^A=\mu_2^A$
on $\Exp(\Exp(A))$. 





 \noindent{\bf{Special  covers:}} We will need the following  covers $\u_l$ of $L$, $l=1,2,\ldots$.
 Let $J_i=[x_i,y_i]$,  $i=0,1,\ldots, 2^l-1$,  be the intervals we obtain in the $l$-th step of the construction of the middle third Cantor set.
 In particular, we have $y_i-x_i=\frac{1}{3^l}$. We assume that $y_i<x_{i+1}$, $i=0,1,\ldots, 2^l-2$.
 Let $b_j=\frac{j}{2^l}$, $j=0,1,\ldots, 2^l$.  
We let
\[ \u_l =\left(\{P(J_i, b_j, b_{j+1}): i=0,1,\ldots, 2^l -1, j=1,\ldots, 2^l -1\}\setminus \{\emptyset\}\right)\cup \{   T(b_1)  \},   \]
where $P(J_i, b_j, b_{j+1})$ is the intersection of $L$ with the 4-gon  determined by lines $y=b_j$, $y=b_{j+1}$, the line passing through
$v$ and $(x_i,1)$ and the line passing through $v$ and $(y_{i},1)$,
and 
$T(b_1)$ is the the intersection of $L$ with the triangle determined by the lines $y=b_1$,  the line passing through $v$ and $(0,1)$, and 
 the line passing through $v$ and $(1,1)$. 
 

 Any  cover $\u_l$ obtained as above will be called a {\em special  cover}.
Note that, by the definition of $\pi$, for any $l$, the open cover $\mathcal{V}_l=\{\pi^{-1}(U): U\in\u_l\}$ consists of disjoint sets,
and when $l\to\infty$, then $\mesh(\u_l)\to 0$ and $\spr(\u_l)\to 0$.

For a special cover $\u_l$, let 
 $A_l=(a_V)_{V\in\mathcal{V}_l}\in\f$ be such that the map $\theta_l:\lel\to A_l$ given by $\theta_l(x)=a_V$ iff  $x\in V$
  is an epimorphism. We will call $\theta_l$ a {\em special epimorphism}.



In the proof of Theorem \ref{equival}, we will use the following lemma.

\begin{lemma}\label{nakr}
Let $A, B\in\f$, let $K$ be the number of elements in $A$, and let $\phi: B\to A$ be an epimorphism such that for every branch $b$ in $B$,
and $a\in A$, if $\phi^{-1}(a)\cap b$ does not contain the endpoint of $b$ then it has at least $2K+1$ elements.
Let $\c$ and $\d$ be maximal chains on $B$ such that $\mu^B(\c,\d)\leq 1$.

Then there is an epimorphism $\psi: B\to A$ such that the maximal chains $\psi(\c)$ and $\psi(\d)$ are equal and for every $a\in A$
\[\psi^{-1}(a)\subset \phi^{-1}(a)\cup \bigcup_{a'\in {\rm is}(a)} \phi^{-1}(a'),\]
where ${\rm is}(a)$ denotes the set of immediate successors of $a$ (and it is a singleton unless $a$ is the root).
\end{lemma}

\begin{proof} 
We will construct $\psi$ by induction on $n\leq K$. 
In step $n$ we will construct $\psi_n:B\to A$ and a downwards closed set $D_n$ of $A$ such that the first $n$ links in both $\psi_n(\c)$ and $\psi_n(\d)$
are equal to $D_1\subset\ldots\subset D_n$. Moreover, for any $a\in A$ and $n\leq K$ it will hold
\begin{equation*}
\begin{split}
\psi_{n}^{-1}(a)\subset \phi^{-1}(a)\cup &\{x: \exists_{ y_0,\ldots, y_m} 
   \text{ such that } m\leq 2(n-1),  \\
   & y_0\in  \phi^{-1}(a), x=y_m
  \text{ and } \forall_{i<m} R^B(y_i, y_{i+1})\}.
  \end{split}
  \end{equation*}
Note that this last containment implies that:

\noindent $(\triangle_n)$ \ \ \ for every branch $b$ in $B$,
and $a\in A$, if $\psi_n^{-1}(a)\cap b$ does not contain the endpoint of $b$ then it has at least $2(K-n+1)+1$ elements.

Then we set $\psi=\psi_K$ and we will have $\psi(\c)=\psi(\d)=\{D_n: n\leq K\}.$ This $\psi$ will be as required.

Observe that by the definition of the metric $\mu^B$, $\mu^B(\c,\d)\leq 1$ means:
for every $C\in \c$ there is $D\in \d$ such that $\mu_1^B(C,D)\leq 1$ and for every
$D\in \d$ there is $C\in \c$ such that $\mu_1^B(C,D)\leq 1$.
 
 {\bf{ Step 1.}} Take $\psi_1=\phi$ and let $D_1$ be the set whose only element is the root of $A$.
 
{\bf{ Step $\mathbf{n+1}$.}}
Let $E=\psi_n^{-1}(D_n)$, $E$ is clearly downwards closed. Let $C\in\c$ be the least (with respect to containment) such that 
there is a  branch $b=(b^i) $ in $B$ and $i_0$ such that $p=b_{i_0}\notin E$ and $p, q=b_{i_0+1}\in C$,
and let $D\in\d$ be such that $\mu_1^B(C,D)\leq 1$.
Take $\psi_{n+1} $ such that $\psi_{n+1}(x)=\psi_n(x)$ for $x\in E\cup(B\setminus (C\cup D))\cup b$. 
 For $x\in (C\cup D)\setminus (E\cup b)$ let
 $c=(c^j)$ be the branch in $B$ such that for some $j_0$, we have $x=c^{j_0}$, and let $\psi_{n+1}(x)=\psi_n(z)$, where $z\in E\cap c$ is the largest with respect to the tree 
 order on $B$. 
 Take $D_{n+1}=D_n\cup\{\psi_n(p)\}$.
Note that $b\cap ((C\cup D)\setminus E)$ consists of 2 or 3 elements, in fact it is equal to $\{p,q\}$ or to $\{p,q, r\}$, where $r$ is the
 immediate successor of $q$. Further 
  by $(\triangle_n)$, $\psi_{n+1}$ is constant on $b\cap ((C\cup D)\setminus E)$.



\end{proof}

 \begin{proof}[Proof of Theorem \ref{equival}]
 It is easy to see that $gH\to g\cdot \c^\lel$ is uniformly continuous. It is straightforward, by the definitions of metrics $d$ and $m$,
 that if $d(gH, hH)<\epsilon$ then
 $m(g\cdot \c^L, h\cdot \c^L)<\epsilon$.
 
 For the opposite direction, we  show that for every $\epsilon$ there is $\delta$ such that for every  
 $g,h\in G$ whenever $m(g\cdot \c^L, h\cdot \c^L)<\delta$, then 
 $d(gH, hH)<3\epsilon$, that is, 
 for some $k\in H$  we have  $d_0(gk, h)<3\epsilon$. This direction requires  more work, we will use Lemma \ref{nakr} and the ultrahomogeneity of $\lel_c$.
Pick $\epsilon>0$ and let $l$ be such that $\mesh(\u)<\epsilon$, where $\u$ is the cover 
$\{U_a\cup\bigcup_{b\in{\rm is}(a)} U_b : a\in A_l\}$ obtained from the cover $\u_l=(U_a)_{a\in A_l}$.

Let $K=|A_l|$ and let  $k\geq (2K+1)l$.
Let $\theta_l:\lel\to A_l$ and $\theta_k:\lel\to A_k$ be 
 special epimorphisms,  and $\phi: A_k\to A_l$ be the epimorphism such that $\phi\theta_k=\theta_l$.
 Denote $A=A_l$ and $B=A_k$.
Take
 $\delta=\spr(\u_k)$ and suppose that
$m(g\cdot \c^L, h\cdot \c^L)<\delta$. 

Suppose first that there are $g_0, h_0\in \aut(\lel)$ such that
 $g=\pi^*(g_0)\in H(L)$  and $h=\pi^*(h_0)\in H(L)$.
Since $m(g\cdot \c^L, h\cdot \c^L)<\spr(\u_k)$, 
we obtain $\mu^B(\theta_k(g_0\cdot\c^{\lel}), \theta_k(h_0\cdot\c^{\lel}) )\leq 1$.
Applying Lemma \ref{nakr} to $\phi:B\to A$, $\c=\theta_k(h_0\cdot\c^{\lel})$ and $\d=\theta_k(g_0\cdot\c^{\lel})$,
we get an epimorphism $\psi: B\to A$ such that $\psi(\c)=\psi(\d)=:\c^A$. 
The conclusion of Lemma \ref{nakr} implies that $\psi\theta_k: (\lel, h_0\cdot\c^{\lel})\to (A,\c^A)$ and $\psi\theta_k: (\lel, g_0\cdot\c^{\lel})\to (A,\c^A)$ are epimorphisms, as well as that  $\psi^{-1}(a)\subset \phi^{-1}(a)\cup \bigcup_{b\in {\rm is}(a)} \phi^{-1}(b)$, for every $a\in A$, 
$\mesh(\u_0)\leq \mesh(\u)<\epsilon$, where $\u_0=\{\pi(( \psi\theta_k)^{-1}(a)): a\in A\}.$
 From the ultrahomogeneity of $\lel_c$ applied to epimorphisms $\psi\theta_k g_0:(\lel, \c^\lel)\to (A,\c^A)$ and 
 $\psi\theta_k h_0:(\lel, \c^\lel)\to (A,\c^A)$,
 we get  $k_0\in \aut(\lel_c)$ such that $\psi\theta_k g_0 k_0=\psi\theta_k h_0$, that is
 for every  $x\in\lel$, $g_0k_0(x)$ and $h_0(x)$ are in the same set of the cover $\{(\psi\theta_k)^{-1}(a): a\in A\}$ of $\lel$.
  This implies that, denoting $k=\pi^*(k_0)$, for every $y\in L$, $gk(y)$ and $h(y)$, 
  are in the same set of the cover $\u_0$, and therefore, since $\mesh(\u_0)<\epsilon$, we get $d_0(gk, h)<\epsilon$, as needed.
  
  In general, if $g,h\in H(L)$,  take $g',h'\in \pi^*(\aut(\lel))\subset H(L)$ such that $d_0(g,g')<\epsilon$
  and $d_0(h,h')<\epsilon$. Then, if $k$ is such that $d_0(g'k,h')<\epsilon$, using the right-invariance of $d_0$ and the triangle inequality, 
  we obtain $d_0(gk,h)<3\epsilon$.
  \end{proof}

\end{document}